\documentclass[12pt]{amsart}
\usepackage[margin=1in,letterpaper,portrait]{geometry}
\usepackage{ytableau}
\ytableausetup{notabloids}
\ytableausetup{centertableaux}
\usepackage{latexsym,amsmath,amssymb,verbatim, tikz}
\usepackage{graphicx}
\usepackage[colorlinks]{hyperref}
\usepackage{mathrsfs}
\usepackage{amsfonts}
\usepackage{amsthm,youngtab}
\usepackage{graphicx}
\usepackage{caption}
\usepackage{enumerate}
\usepackage{tikz-cd}
\usepackage{rotating}
\usepackage[all]{xy}
\usepackage[titletoc, title]{appendix}
\usepackage{amscd}
\usepackage{float}
\hypersetup{
bookmarksnumbered,
pdfstartview={FitH},
breaklinks=true,
linkcolor=blue,
urlcolor=blue,
citecolor=blue,
bookmarksdepth=2
}	
\usepackage{caption}
\usepackage[utf8x]{inputenc}
\usepackage{adjustbox}

\begin{document}

\theoremstyle{plain}
   \newtheorem{thm}{Theorem}[section]
   \newtheorem{prop}[thm]{Proposition}
   \newtheorem{lem}[thm]{Lemma}
   \newtheorem{cor}[thm]{Corollary}
   \newtheorem{conj}[thm]{Conjecture}
   \newtheorem{problem}[thm]{Problem}
\theoremstyle{definition}
   \newtheorem{deftn}[thm]{Definition}
   \newtheorem{example}[thm]{Example}
   \newtheorem{examples}[thm]{Examples}
   \newtheorem{question}[thm]{Question}
   \newtheorem{algorithm}[thm]{Algorithm}
   \newtheorem{rk}[thm]{Remark}
   \newtheorem{obs}[thm]{Observation}

\numberwithin{equation}{section}

\newcommand\comp[2]{\alpha{(#2,#1)}}
\newcommand\ccomp[2]{\alpha^\cyc{(#2,#1)}}

\newcommand{\CC}{{\mathbb {C}}}
\newcommand{\QQ}{{\mathbb {Q}}}
\newcommand{\RR}{{\mathbb {R}}}
\newcommand{\ZZ}{{\mathbb {Z}}}
\newcommand{\NN}{{\mathbb {N}}}

\newcommand{\EE}{{\mathcal{E}}}
\newcommand{\PP}{{\mathcal {P}}}

\newcommand{\Cyl}{\operatorname{Cyl}}
\newcommand{\Tor}{\operatorname{Tor}}

\newcommand{\then}{\Rightarrow}

\newcommand{\Aut}{{\operatorname{Aut}}}
\newcommand{\Des}{{\operatorname{Des}}}
\newcommand{\cDes}{{\operatorname{cDes}}}
\newcommand{\cyc}{{\operatorname{cyc}}}
\newcommand{\cdes}{{\operatorname{cdes}}}
\newcommand{\des}{{\operatorname{des}}}
\newcommand{\ch}{{\operatorname{ch}}}
\newcommand{\shape}{{\operatorname{shape}}}
\newcommand{\SYT}{{\operatorname{SYT}}}
\newcommand{\co}{{\operatorname{co_n}}}
\newcommand{\cc}{{\operatorname{cc}_n}}
\newcommand{\SSYT}{{\rm SSYT}}
\newcommand{\tchi}{{\widetilde{\chi}}}
\newcommand{\td}{{\widetilde{D}}}
\newcommand{\bd}{{\bar{D}}}
\newcommand{\klr}{{\rm KLR}}

\newcommand{\Comp}{{\operatorname{Comp}}}
\newcommand{\cComp}{{\operatorname{cComp}}}
\newcommand{\ccon}{{\operatorname{cyc-comp}(n)}}

\newcommand{\one}{{\mathbf{1}}}

\newcommand{\cs}{{\tilde{s}}}

\newcommand{\xx}{{\mathbf{x}}}
\newcommand{\ttt}{{\mathbf{t}}}
\newcommand{\symm}{{\mathfrak{S}}}

\def\AA{\mathbb{A}}

\newcommand{\AAA}{{\mathcal{A}}}
\newcommand{\EEE}{{\mathcal{E}}}
\newcommand{\CCC}{{\mathcal{C}}}
\newcommand{\LLL}{{\mathcal{L}}}
\newcommand{\MMM}{{\mathcal{M}}}
\newcommand{\OOO}{{\mathcal{O}}}
\newcommand{\TTT}{{\mathcal{T}}}
\newcommand{\shihao}{\fontsize{3.25pt}{\baselineskip}\selectfont}

\newcommand{\stirling}[2]{\left\{\begin{matrix}{#1}\\ {#2}\end{matrix}\right\}}

\DeclareRobustCommand{\stirlingI}{\genfrac \langle \rangle {0pt}{}}
\newcommand{\GaussBinomial}[2]{\left[\begin{matrix}{#1}\\ {#2}\end{matrix}\right]}

\newcommand{\SteinbergTorus}{{\widetilde{\bdelta}}}

\newcommand{\wcDes}{\cDes_*}
\newcommand{\wcdes}{\cdes_*}

\newcommand{\AffDes}{{\operatorname{cDes}}}

\newcommand{\udots}{\reflectbox{$\ddots$}}

\newcommand{\bT}{{\mathbf T}}
\newcommand{\bA}{{\mathbf A}}
\newcommand{\tensor}{\otimes}
\newcommand{\bPsi}{\boldsymbol\Psi}
\newcommand{\kk}{\Bbbk}
\newcommand{\lie}[1]{{\mathfrak{#1}}}

\newcommand\scalemath[2]{\scalebox{#1}{\mbox{\ensuremath{\displaystyle #2}}}}

\newlength{\mysizetiny}
\setlength{\mysizetiny}{0.3em}
\newlength{\mysizesmall}
\setlength{\mysizesmall}{0.8em}
\newlength{\mysize}
\setlength{\mysize}{1.3em}
\newlength{\mysizelarge}
\setlength{\mysizelarge}{2em}

\newcommand{\IZ}{\widehat{I}_{\mathbb{Z}}}
\newcommand{\Ixi}{\widehat{I}_{\leq \xi}}
\newcommand{\YZ}{\mathcal{Y}_{\mathbb{Z}}}
\newcommand{\tdxi}{\widetilde{D}_{\xi}}
\newcommand{\barD}{\overline{D}}
\newcommand{\barDstar}{\overline{D}^{*}}
\newcommand{\xik}{x_k^{\mathbf{i}}}
\newcommand{\xkop}{x_k^{op}}
\newcommand{\betakop}{\beta_k^{op}}
\newcommand{\CN}{\mathbb{C}[\mathbf{N}]}
\newcommand{\si}{\mathcal{S}^{\mathbf{i}}}
\newcommand{\ihat}{\widehat{\mathbf{i}}}
\newcommand{\HQ}{\mathcal{H}_Q}
\newcommand{\AQ}{\mathcal{A}_Q}
\newcommand{\MQ}{\mathcal{M}_Q}
\newcommand{\tMQ}{\widetilde{\mathcal{M}_Q}}
\newcommand{\ttMQ}{\widetilde{\widetilde{\mathcal{M}_Q}}}
\newcommand{\DQ}{\mathcal{D}_Q}
\newcommand{\IQ}{\mathcal{I}_Q}
\newcommand{\KQ}{\mathcal{K}_Q}
\newcommand{\TQ}{\mathcal{T}_Q}
\newcommand{\tCQ}{\widetilde{\mathcal{C}_Q}}
\newcommand{\ZQ}{\mathbb{Z}Q}
\newcommand{\ZQz}{(\mathbb{Z}Q)_0}
\newcommand{\CZ}{\mathcal{C}_{\mathbb{Z}}}
\newcommand{\CxiZ}{\mathcal{C}_{\mathbb{Z}}^{\leq \xi}}
\newcommand{\Cxi}{\mathcal{C}_{\xi}}
\newcommand{\YxiZ}{\mathcal{Y}_{\mathbb{Z}}^{\leq \xi}}
\newcommand{\CxiZstar}{\mathcal{C}_{\xi^{*}}^{\mathbb{Z}}}
\newcommand{\barCQ}{\overline{\mathcal{C}_Q}}
\newcommand{\tdKxi}{\widetilde{K}_{\xi}}
\newcommand{\tdKxistar}{\widetilde{K}_{\xi^{*}}}
\newcommand{\DbRepQ}{\mathcal{D}^b(\mathrm{Rep}Q)}
\newcommand{\xistar}{\xi^{*}}
\newcommand{\bkZQ}{\overline{k(\mathbb{Z}Q)}}
\newcommand{\xbar}{\overline{x}}
\newcommand{\ybar}{\overline{y}}
\newcommand{\tH}{\widetilde{\mathcal{H}}}
\newcommand{\ObQ}{\mathrm{Ob}_Q}
\newcommand{\ObCQ}{\mathrm{Ob} \left( \mathcal{C}_Q \right)}
\newcommand{\tHi}{\widetilde{H}_i}
\newcommand{\tWi}{\widetilde{W}_i}
\newcommand{\tHj}{\widetilde{H}_j}
 \newcommand{\bYx}{Y_{\bullet}(x)}
  \newcommand{\bFx}{F_{\bullet}(x)}
 \newcommand{\HQG}{\mathcal{C}_Q^{\Gamma}}
 \newcommand{\IG}{\mathcal{J}^{\Gamma}}
 \newcommand{\KQG}{\mathcal{K}_Q^{\Gamma}}
\newcommand{\mbeta}{\mathfrak{m}_{\beta}}
\newcommand{\bdelta}{\boldsymbol{\delta}}
\newcommand{\bomega}{\boldsymbol{\omega}}
\newcommand{\Qbeta}{Q_{\beta}}

   \newcommand{\Ybeta}{Y[\beta]}
   \newcommand{\Cbeta}{C_{\bullet}[\beta]}

    \newcommand{\odi}{\overline{\mathbf{d}_i}}
    \newcommand{\udi}{\underline{\mathbf{d}_i}}

\title[Triangulated monoidal categorifications]{Triangulated monoidal categorifications \\ of finite type cluster algebras}
\author{\'Elie Casbi}
\address{Oskar-Morgenstern Platz 1, 1090 Wien, Österreich}
\email{elie.casbi@univie.ac.at}
\date{}

 \begin{abstract}
 We propose a framework of monoidal categorification of finite type cluster algebras involving triangulated monoidal categories. Namely, given a Dynkin quiver $Q$, we consider the bounded homotopy category $\KQ^{(1)}$ of a symmetric monoidal category $\HQ^{(1)}$ that we define in terms of the Auslander-Reiten theory of $Q$. Using some iterated mapping cone procedure,  we construct a  distinguished family  $\{ \Cbeta \}_{\beta \in \Delta_+}$ of chain complexes in $\KQ^{(1)}$ characterized (up to isomorphism) by homological conditions similar to those of higher exact sequences appearing in the context of higher homological algebra. We then prove that the distinguished triangle in $\KQ^{(1)}$ given by each mapping cone categorifies an exchange relation in the finite type cluster algebra $\AQ$ with initial exchange quiver $Q$ (for a suitable choice of frozen variables). As a consequence, we obtain that for each positive root $\beta$, the Euler characteristic of $\Cbeta$ coincides with the truncated $q$-character of the simple module $L[\beta]$ in the HL category $\Cxi^{(1)}$ categorifying the cluster variable $x[\beta]$ of $\AQ$ via Hernandez-Leclerc's monoidal categorification. 
 Along the way, we establish a uniform formula for the dominant monomial of $L[\beta]$ in all types $A_n$ and $D_n$ for arbitrary orientations (agreeing with Brito-Chari's results in type $A_n$). 
 \end{abstract}

\maketitle

\section{Introduction}

 Monoidal categorifications of cluster algebras have been involved in several major developments over the past fifteen years. They were first introduced by Hernandez-Leclerc \cite{HL10} as a way to provide new insights in cluster theory using the representation theory of quantum affine algebras. Following \cite{HL10}, a monoidal categorification of a cluster algebra $\mathcal{A}$ consists of an abelian monoidal category $\mathcal{C}$ such that there is a ring isomorphism $\mathcal{A} \simeq K_0(\mathcal{C})$ from $\mathcal{A}$ to the Grothendieck ring of $\mathcal{C}$ sending the cluster monomials in $\mathcal{A}$ to isomorphism classes of real simple objects in $\mathcal{C}$  (real meaning of simple tensor square). In this framework, the exchange relations associated to cluster mutations are interpreted as identities in the Grothendieck ring of $\mathcal{C}$ coming from certain distinguished short exact sequences in $\mathcal{C}$. Moreover, the cluster expansion of a cluster monomial with respect to a well-chosen initial cluster can be identified with the $q$-character of  the corresponding (real) simple module under the isomorphism $\mathcal{A} \simeq K_0(\mathcal{C})$. The representation theory of quantum affine algebras (and recently shifted quantum affine algebras) provides powerful tools to construct abelian monoidal categorifications of Lie-theoretic cluster algebras. These categorifications can be viewed as subcategories of a certain monoidal category $\CZ$  introduced by Hernandez-Leclerc \cite{HL15}, which is a monoidal subcategory of the category of finite-dimensional representations of the quantum affine algebra associated to $\mathfrak{g}$.   
 Although Hernandez-Leclerc's original work \cite{HL10} mostly dealt with cluster algebras $\mathcal{A}_{Q^{bip}}$ of finite type $A_n$ associated with a bipartite initial exchange quiver, their results have later been extended to more sophisticated cluster structures. As far as finite type cluster algebras are concerned, Brito-Chari  \cite{BC} constructed monoidal categorifications $\Cxi^{(1)} \subset \CZ$ of any finite type cluster algebra $\AQ$ arising from an arbitrarily oriented initial exchange quiver $Q$ of type $A_n$ (here $\xi$ denotes a height function adapted to $Q$). It was later showed by Kashiwara-Kim-Oh-Park \cite{KKOPCompos20,KKOP1,KKOP2} that various subcategories of $\CZ$ were monoidal categorifications of interesting cluster structures, including the category $\Cxi^{(1)}$ in all Dynkin types and for arbitrary orientations. 
 
  The constructions of monoidal categorifications known so far can essentially be grouped into two groups. On the one hand, categories of modules over associative algebras such as (shifted) quantum affine algebras \cite{HL10,BC,KKOPCompos20} or quiver Hecke (KLR) algebras  \cite{KKKO}. On the other hand, categories of perverse sheaves on algebraic varieties naturally appearing in geometric representation theory such as Nakajima quiver varieties or affine Grassmannians \cite{Naka,CW}.
  In both cases the  categories involved are in particular abelian categories. In the present work, we propose an alternative framework relying on triangulated monoidal categories, namely,  homotopy categories of certain (symmetric) monoidal categories. The main motivation for the construction proposed in the present work stems from our desire to provide a conceptual interpretation of our recent results with Jian-Rong Li \cite{CasbiLi2}. 

 In our previous work \cite{CasbiLi}, given a simply-laced type Lie algebra $\mathfrak{g}$ and an orientation $Q$ of its Dynkin diagram, we introduced  an algebra homomorphism $\td_Q : \YZ \longrightarrow \mathbb{C}[\mathfrak{t}^{reg}]$ which is intimately related to the map $\barD$ defined in \cite{BKK} by Baumann-Kamnitzer-Knutson  in their study of the equivariant homology of Mirkovi\'c-Vilonen cycles. Here $\YZ$ denotes the ring of Laurent polynomials containing the $q$-characters of all objects in the Hernandez-Leclerc's $\CZ$.  In our following work \cite{CasbiLi2}, we used the map $\td_Q$ to exhibit new families of non-trivial rational identities out of the representation theory of quantum affine algebras. More precisely, we showed that for any standard module $M$ in $\CZ$, we have 
 $$ \td_Q \left( \chi_q(M) \right) = 0 .  $$
  The starting point of the present work consists in viewing these identities as the vanishing of certain alternate sums of dimension vectors (after getting rid of the denominators). In other words, to each object in $\CZ$ should correspond a chain complex satisfying some kind of exactness property, and whose objects should keep track (in a bijective way) of the Laurent monomials appearing in the $q$-character of $M$. This last condition shall be ensured by using the interpretation of $q$-characters as cluster expansions with respect to well-chosen cluster structures (which is one of the key upshots provided by Hernandez-Leclerc's monoidal categorifications). Note that certain special families of exact chain complexes (higher Auslander-Reiten sequences) involving tensor products of modules over the path algebra $\mathbb{C}Q$ have been considered in \cite{Pasquali17}. However, they do not seem to be  appropriate for our needs. 
  
  The above discussion is our original motivation for introducing the framework of triangulated monoidal categorification of cluster algebras, whose main idea  can be summarized as follows. Given a finite acyclic quiver $Q$, we begin by constructing a symmetric $\mathbf{k}$-linear monoidal category $\HQ$ out of the Auslander-Reiten theory of $Q$, whose objects and morphisms are defined in a way to behave similarly as tensor products (over $\mathbf{k}$) of modules over the path algebra $\mathbf{k}Q$. Then, considering the bounded homotopy category $\mathcal{K}$ of a suitably chosen monoidal subcategory $\mathcal{C}$ of $\HQ$, one shall construct chain complexes in $\mathcal{K}$ characterized  by suitable exactness properties. Namely, starting from an initial collection of chain complexes concentrated in a single homological degree, these non-trivial chain complexes shall be constructed through some iterated mapping cone procedure (involving the monoidal structure of $\HQ$), which crucially relies on the aforementioned  exactness conditions. Each  mapping cone involved in this process yields a  distinguished triangle in $\mathcal{K}$ that shall categorify some exchange relation associated with a cluster mutation in  an appropriate cluster algebra $\mathcal{A}$. Consequently, the Euler characteristics of the constructed chain complexes shall coincide with certain cluster characters in $\mathcal{A}$.

  We now briefly present the  monoidal category $\HQ$ that will be playing the role of ambient category throughout this paper. Given a finite acyclic quiver $Q$, let us denote by $\ZQ$ the repetition quiver of $Q$ and by $\DbRepQ$ the bounded derived category of the category of finite dimensional representations of $Q$ over the field of complex numbers. We first construct a certain  $\mathbb{C}$-linear symmetric monoidal category $\MQ$, whose indecomposable objects  are given by pairs $(X,h)$ where $X$ is a multiset of indecomposable objects in $\DbRepQ$ and $h : \ZQz \rightarrow \mathbb{Z}$ is a function that will usually be required to satisfy a certain property called \textbf{quasi-additivity} (cf. Definition~\ref{def : quasi-additive functions}). Quasi-additive functions generalize the well-known combinatorics of  additive functions on repetition quivers studied by Gabriel \cite{Gabriel}, Brenner \cite{Brenner} and Ringel \cite{Ringel}.  We define a certain operation on such pairs $(X,h)$, called \textbf{Serre tilting}, which consists in replacing finitely many elements of $X$ by their image under the standard Serre functor of $\DbRepQ$, as well as modifying  $h$ in a suitable way. Then, given two multisets $X = \{x_k, k \in  K\}$ and $Y = \{y_l , l \in L\}$ and two quasi-additive  functions $f$ and $g$, the space of  morphisms between $(X,f)$ and $(Y,g)$ is non-trivial only if $(Y,g)$ is a Serre tilting of $(X,f)$, and if so, then it is given by
 \begin{equation}
     \bigoplus_{  \substack{ \sigma \in \mathrm{Bij}(K,L) \\  \sharp \mathrm{Supp}_{X,Y}(\sigma) < \infty }} \bigotimes_{k \in \mathrm{Supp}_{X,Y}(\sigma)} \mathrm{Hom}_{\DbRepQ}(x_k , y_{\sigma(k)}) 
 \end{equation}
where $\mathrm{Bij}(K,L)$ denotes the set of bijections from $K$ to $L$ and $\mathrm{Supp}_{X,Y}(\sigma) := \{k \in K \mid y_{\sigma(k)} \neq x_k \}$ for any $\sigma \in \mathrm{Bij}(K,L)$. The category $\HQ$ that will serve as underlying category for the framework proposed in this paper will be a certain monoidal full subcategory of $\MQ$ whose indecomposable objects will be given by appropriate Serre tiltings of certain distinguished objects $Y(x) := (H(x),h_x)$, $ x \in \ZQz$ in $\MQ$ called \textbf{hammock objects}. For each $x \in \ZQz$ (viewed as the isomorphism class of an indecomposable object in $\DbRepQ$ via Happel's theorem), the multiset $H(x)$  is obtained by considering the codomains of all non-trivial morphisms in $\DbRepQ$ having domain $x$, while the function $h_x$ is essentially a hammock function (hence the terminology of hammock object) considered by Gabriel, Brenner and Ringel \cite{Gabriel, Brenner,Ringel}. Objects in $\HQ$ that are isomorphic to tensor products of hammock objects will be called \textbf{dominant objects}.

  In this paper, we will be mostly working on constructing  triangulated monoidal categorifications of finite type cluster algebras $\AQ$ with arbitrarily oriented initial exchange quiver $Q$. For this purpose, we will be focusing on a certain additive monoidal (but not full) subcategory $\HQ^{(1)}$ of $\HQ$ that depends only on the choice of a section of $Q^{op}$ inside $\ZQ$. We will then  consider the monoidal triangulated category  $\KQ^{(1)}$ defined as the bounded homotopy category of $\HQ^{(1)}$. A key property that will eventually characterize the objects in $\KQ^{(1)}$ categorifying the cluster variables in $\AQ$  is the notion that we call  \textbf{$\HQ^{(1)}$-exactness}. It can be viewed as a slightly weakened version of the notion of $n$-exactness of chain complexes appearing in higher homological algebra \cite{Iyama07,Jasso16,GKO13}. 
  Given a bounded chain complex $C_{\bullet}$, this exactness property consists in requiring the following conditions: firstly, we should have $C_n=0$ for all $n<0$ and $C_0$ should be a dominant object in $\HQ^{(1)}$; secondly (and most importantly), a certain class of morphisms whose composition with $d_{n-1}^{C_{\bullet}}$ vanishes are required to factor through $d_n^{C_{\bullet}}$. This extra condition is essentially combinatorial and entirely  depends on $C_0$. We also introduce a slightly stronger version (strong $\HQ^{(1)}$-exactness) designed to ensure the uniqueness up to isomorphism of the considered chain complex. We refer to Definitions~\ref{def : CQ1 exact chain complexes} and~\ref{def : strong exactness} for more details. 
  Of particular importance will be certain dominant objects $\Ybeta$ in $\HQ^{(1)}$ parametrized by elements of the positive orthant of the root lattice. 
  The first main result of this paper will be the following.

  \begin{thm} \label{thm : thm1 intro}
      For each positive (or negative simple) root $\beta$, there exists a unique (up to isomorphism) strongly $\HQ^{(1)}$-exact complex $C_{\bullet}[\beta]$ in $\KQ^{(1)}$ such that $C_n[\beta]=0$ if $n<0$ and $C_0[\beta]$ is indecomposable and isomorphic to $Y[\beta]$. 
  \end{thm}

 A key ingredient in the proof of Theorem~\ref{thm : thm1 intro} consists in establishing the existence of certain morphisms having weak cokernels in $\HQ^{(1)}$ (we refer to Theorem~\ref{thm : existence of morphisms eta} below for a precise statement). Let us also mention that evidences suggest $\Cbeta$ is indecomposable in $\mathcal{K}^b(\HQ^{(1)})$ but we will not tackle this question in detail in this paper.

 The second main result of this paper will establish the connection between the triangulated monoidal category $\KQ^{(1)}$ and the cluster algebra $\AQ$. Namely, each cluster variable in $\AQ$ shall be categorified by a $\HQ^{(1)}$-exact chain complex in $\KQ^{(1)}$, in the sense that the cluster expansion of this cluster variable with respect to a well-chosen cluster will be identified with the Euler characteristic of this chain complex. 
In order to make this statement more precise, we will denote by $\chi(C_{\bullet})$ the Euler characteristic of a chain complex  $C_{\bullet} = \cdots \rightarrow C_0 \rightarrow C_1 \rightarrow \cdots $  in  $\KQ^{(1)}$, which is given by 
 $$ \chi \left( C_{\bullet} \right) := \sum_{n \in \mathbb{Z}} (-1)^n [C_n] \enspace \in K_0(\HQ^{(1)})  $$
  where   $K_0(\HQ^{(1)})$ denotes the split Grothendieck group of $\HQ^{(1)}$.  On the other hand, denote by $L[\beta]$ the real simple module in Hernandez-Leclerc's category $\Cxi^{(1)}$ that categorifies the cluster variable $x[\beta]$ in $\AQ$, and denote by $\tchi_q(L[\beta])$ its truncated $q$-character (cf. \cite{HL10}). We then have the following.

\begin{thm} \label{thm : thm2 intro}
Assume $Q$ is a Dynkin quiver of type $A_n$ with $n \geq 1$ or $D_n$ with $n \geq 4$.
Then there is an injective ring homomorphism 
   $$  \iota_Q : \enspace  K_0(\HQ^{(1)}) \longrightarrow \YZ[F_i , i \in I] $$
   such that for each positive root $\beta$, we have 
    $$  \iota_Q \left( \chi \left( C_{\bullet}[\beta] \right) \right)_{\mid F_i = -1}  =  \tchi_q \left( L[\beta] \right) .  $$
\end{thm}

 While proving Theorem~\ref{thm : thm2 intro}, we obtain a uniform formula for the dominant monomial of $L[\beta]$, for arbitrary Dynkin quivers of  type $A_n$ or $D_n$, which we believe was not known so far. Note that we rely on the fact that all cluster monomials in $\AQ$ are classes of simple modules in $\Cxi^{(1)}$, which follows from results by Kashiwara-Kim-Oh-Park \cite{KKOP1,KKOP2}. Our formula agrees with Brito-Chari's results \cite{BC} in type $A_n$. 
 
 It is worth outlining that although Theorems~\ref{thm : thm1 intro} and~\ref{thm : thm2 intro} deal with finite type cluster structures, the constructions proposed in this paper are designed to provide triangulated monoidal categorifications of other kinds of cluster structures, such as the ones appearing in representation theory  \cite{HL16,KKKO,KKOP1} or alternatively by considering underlying quivers $Q$ of tame or wild type. This may open perspectives in several directions, some of which we mention below.
   
    Firstly, as shown by the statement of Theorem~\ref{thm : thm2 intro}, the $q$-characters of finite dimensional representations of affine quantum groups may be identified with the Euler characteristics of certain chain complexes whose objects live in a category entirely governed by the representation theory of the quiver $Q$. Exhibiting such intimate connections between Auslander-Reiten theory and the representation theory of quantum affine algebras may shed some light on certain combinatorial observations relating numerical invariants respectively appearing in these two frameworks (such as our results of \cite{CasbiLi2} for example) and may yield new interpretations of the coefficients of certain $q$-characters. 

     Most importantly, we expect the triangulated framework proposed in this paper to be relevant for exhibiting monoidal categorifications of  other cluster structures than those appearing in Lie theory. Indeed, starting from a certain family of chain complexes concentrated in one single homological degree, the iterated mapping cone procedure used to construct the chain complexes categorifying cluster variables only requires the existence of certain morphisms having weak cokernels and whose compositions should satisfy certain vanishing conditions. This is of course non trivial but can potentially hold in categories arising from various contexts. 

      Finally, several features from our construction share close similarities with notions appearing in the theory of additive categorifications of cluster algebras; for instance, the notion of Serre tilting (cf. Definition~\ref{def : Serre tiltings} below) that plays a central role in our construction. More conceptually,  as indicated by Theorem~\ref{thm : thm1 intro}, the chain complexes categorifying cluster variables in our framework satisfy certain properties that are strikingly similar to those appearing in the contexts of $n$-abelian or $n$-angulated categories  \cite{Jasso16,GKO13}. 
    Thus, our construction may allow to produce large families of new non-trivial examples of $n$-exact sequences (potentially including the case $n=0$  recently developed in \cite{Gulisz25}) suggesting the existence of deep connections with higher homological algebra. In turn, this may provide new insights on the links between additive and monoidal categorifications of cluster algebras.

     The paper is organized as follows.  Sections~\ref{sec : background} and~\ref{sec : reminders on monoid categ} are respectively devoted to the necessary recollections on Auslander-Reiten theory and (abelian) monoidal categorifications of cluster algebras in the sense of Hernandez-Leclerc. Section~\ref{sec : the underlying categ}  contains the key definitions concerning the ambient category $\HQ$ lying at the heart of the framework proposed in this paper. In Section~\ref{sec : restriction to the section}, we focus on the   subcategory $\HQ^{(1)}$ which is appropriate to deal with finite type cluster algebras. Section~\ref{sec : exact chain complexes} is devoted to the main homological definitions  as well as the  construction of the distinguished chain complexes $\Cbeta$ in the bounded homotopy category of $\HQ^{(1)}$. Finally, in Section~\ref{sec : q characters as Euler char} we prove the second main result of this paper, relating the Euler characteristics of the chain complexes $\Cbeta$ to the (truncated) $q$-characters of simple modules in the  Hernandez-Leclerc's category $\Cxi^{(1)}$.

 \subsection*{Acknowledgements}
  I would like to express  all my gratitude to Gordana Todorov and Bernhard Keller for patiently providing me their insights and answering my questions about Auslander-Reiten theory. I also thank Masaki Kashiwara, Pavel Etingof,  Bernard Leclerc, David Hernandez,  Daniel Labardini-Fragoso, Jian-Rong Li and Keyu Wang for fruitful discussions. Finally, I warmly thank Vitor Gulisz for making very useful suggestions on early versions of this work.

 \section{Auslander-Reiten theory} \label{sec : background}

  \subsection{Auslander-Reiten theoretic background}
 \label{sec : reminders on AR}

  Let $Q$ be a (finite) acyclic quiver without loops and multiple edges. Let us denote by $I := Q_0$ its set of vertices. We denote by $\ZQ$ the repetition quiver of $Q$, whose set of vertices is  $\ZQz := \{(i,p) , p \in i + 2 \mathbb{Z}\}$ and whose set of arrows is given by 
  $$ \forall \alpha \in Q_1,  \quad (t(\alpha),p) \rightarrow (s(\alpha),p+1) \enspace \text{and} \enspace (s(\alpha),p) \rightarrow (t(\alpha),p+1) .  $$
  Here $s(\alpha)$ and $t(\alpha)$ respectively stand for the source and target of an arrow $\alpha \in Q_1$. By definition, a section of $Q$ in $\ZQ$ refers to any subset of $\ZQ$ of the form $\{(i,p_i) , i \in I\}$ such that $p_{s(\alpha)} - p_{t(\alpha)} = 1$ for every arrow $\alpha$ in $Q$. In particular, given any vertex $x \in \ZQz$, there is a unique section of $Q$ in $\ZQ$ containing $x$. We shall denote it by $Q_x$. 
   We will denote by $\tau$ the Auslander-Reiten translate on $\ZQ$ which is given by $\tau (i,p) := (i,p-2)$.
   In all what follows we will be working over an algebraically closed field $\mathbf{k}$ of characteristic zero. We denote by $\mathrm{Rep}Q$ the category of finite dimensional representations of $Q$ over $\mathbf{k}$ (equivalently, finite-dimensional modules over the path algebra $\mathbf{k}Q$ of $Q$) and by $\DbRepQ$ the bounded derived category of $\mathrm{Rep}Q$. It follows from Happel's theorem (Proposition 4.6 in \cite{Happel}) that there is an injective set-theoretic map $j_Q$ from $\ZQz$ to the set of isomorphism classes of indecomposable objects in $\DbRepQ$. This injection is moreover a bijection in the case where $Q$ is a Dynkin quiver, i.e. an orientation of the Dynkin diagram of a simple complex Lie algebra of type $A_n , n \geq 1 , D_n , n \geq 4$ or $E_n , n = 6,7,8$. We will denote 
   $$ M_{i,p} := j_Q(i,p) $$
   for each $(i,p) \in \ZQz$ and will simply identify a vertex $x$ of $\ZQ$ with its image under $j_Q$  when the explicit data of $i$ and $p$ are not needed. 
    In particular, we shall denote by $\Sigma x$ the element of $\ZQz$ given by $j_Q^{-1}(\Sigma j_Q(x))$, where $\Sigma$ denotes the standard shift in $\DbRepQ$.  Denoting by $\tau$  the Auslander-Reiten  translate functor in $\DbRepQ$, we set 
    $$ S := \Sigma \tau = \tau \Sigma . $$
    It follows from the Auslander-Reiten formula that $S$ is a Serre functor in $\DbRepQ$, i.e. 
    \begin{equation} \label{eq : isom def of Serre functor}
         \forall M,N \enspace \mathrm{Hom}(M,N) \simeq D \mathrm{Hom}(N,SM) 
      \end{equation}    
    where $D$ denotes the usual duality of $\mathbf{k}$-vector spaces. 
Recall the notion of an irreducible morphism as introduced by Auslander and Reiten in \cite{AR4}: a morphism $f : X \to Y$ in a given category is called \textit{irreducible} if $f$ is neither a split monomorphism nor a split epimorphism, and for any commutative diagram \[ \begin{tikzcd}
                                   & W \arrow[rd, "h"] &   \\
X \arrow[rr, "f"'] \arrow[ru, "g"] &                   & Y
\end{tikzcd} \] either $g$ is a split monomorphism or $h$ is a split epimorphism. (recall that a morphism $u$ is a split monomorphism when there is some $u'$ such that $u'u = 1$, and that $u$ is a split epimorphism when there is some $u'$ for which $uu' = 1$).

  \subsection{Additive functions on repetition quivers} \label{sec : additive functions}

  Given a function $f : (\ZQ)_0 \longrightarrow \mathbb{Z}$, we set 
  $$ \widetilde{f} : x \in (\ZQ)_0 \longmapsto f(x) + f(\tau x) - \sum_{ y \rightarrow x} f(y) . $$
  Following Gabriel \cite{Gabriel}, a function $f$ is said to be \textbf{additive} if $\widetilde{f}=0$. It is a known fact (cf. \cite{Gabriel}) that an additive function is uniquely determined by its values on any given section of $Q$ in $\ZQ$. In other words, given $x \in \ZQz$ and a $n$-tuple of integers $(a_1, \ldots , a_n)$, there is a unique additive function $f$ satisfying $f(x_j)=a_j$ for every $1 \leq j \leq n$, where $x_j , j \in I$ denote the vertices of the section $Q_x$. Of particular interest are the \textbf{hammock functions} studied by Gabriel \cite{Gabriel} and Brenner \cite{Brenner} (later generalized by Ringel \cite{Ringel}) defined as the unique additive function satisfying $f(y) = \dim \mathrm{Hom}_{\DbRepQ}(x,y)$ for each $y \in Q_x$. Equivalently, it is the additive function corresponding to the tuple $(a_1, \ldots , a_n)$ defined by $a_j=1$ if there is an oriented path from $j$ to $i$ in $Q$, $a_j = 0$ otherwise (here $i$ denotes the vertex of $Q$ such that $x = (i,p) ,  p \in i + 2 \mathbb{Z}$).

\begin{example} \label{ex : additive function}
Consider the Dynkin quiver $Q$ of type $A_4$ given by $1 \rightarrow 2 \rightarrow 3 \leftarrow 4$. We picture below the hammock function corresponding to the vertex  of $\ZQ$ outlined in red. 

 $$ \xymatrixrowsep{1pc} \xymatrixcolsep{1pc} 
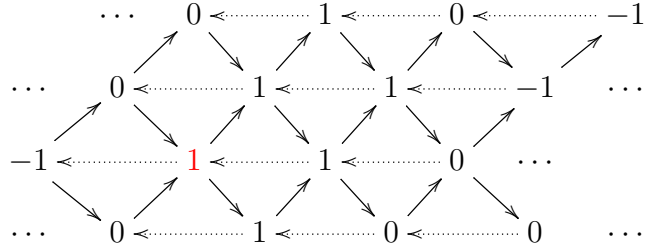
\xymatrix{  
   {} &  \cdots   & 0 \ar[rd]   &    {}   &    1  \ar@{.>}[ll] \ar[rd]   & {}  & 0 \ar@{.>}[ll] \ar[rd] & {} & -1 \ar@{.>}[ll]   \\
    \cdots    &  0 \ar[ru] \ar[rd]   &    {}      &  1 \ar@{.>}[ll]  \ar[ru]  \ar[rd]  &{}  &  1 \ar@{.>}[ll] \ar[ru]  \ar[rd] & {} & -1 \ar@{.>}[ll]  \ar[ru] & \cdots \\
      -1 \ar[ru] \ar[rd]  &     {}     &    \textcolor{red}{1} \ar[ru]  \ar@{.>}[ll]   \ar[rd]   & {} & 1 \ar@{.>}[ll]  \ar[rd] \ar[ru] & {} & 0 \ar[ru] \ar@{.>}[ll] \ar[rd] & \cdots  \\
    \cdots   &   0  \ar[ru]   &  {}       &   1 \ar@{.>}[ll]  \ar[ru]   & {} & 0  \ar@{.>}[ll]   \ar[ru]& {} & 0 \ar@{.>}[ll] & \cdots 
    } $$

\end{example}

 \section{Monoidal categorifications}
 \label{sec : reminders on monoid categ}

 \subsection{Finite type cluster algebras associated with Dynkin quivers}
  \label{sec : reminders AQ}

 Given a Dynkin quiver $Q$, we will be denoting by $\AQ$ the cluster algebra with initial seed
 $$  \left( (x_1, \ldots , x_n , X_1 , \ldots, X_n)  , \overline{Q} \right)$$
 where the initial exchange quiver $\overline{Q}$ is given as follows. The set of vertices of $\overline{Q}$ is $I \times \{0,1\}$ (where $I$ denotes the set of vertices of $Q$) where $I \times \{0\}$ is the set of mutable vertices while $I \times \{1\}$ is the set of frozen vertices. The set of arrows consists in arrows $(i,1) \leftarrow (i,0) , i \in I$ (the horizontal arrows), as well as $(i,0) \leftarrow (j,0) , (i,1) \leftarrow (j,1)$ and  $(i,1) \rightarrow (j,0)$ if there is an arrow $i \rightarrow j$ in $Q$ (the latter arrows are sometimes called ordinary arrows). 

  \smallskip

  It follows from Fomin-Zelevinsky's results \cite{FZ2} that $\AQ$ is a finite type cluster algebra, i.e. $\AQ$ contains finitely many  clusters (equivalently finitely many cluster variables). Consequently, the non-frozen cluster variables in $\AQ$ are in bijection with the set of almost positive roots which is defined as $\bdelta_+ \cup \Pi_-$. We will denote by $x[\beta]$ the cluster variable corresponding to $\beta$ under this bijection, for each almost positive root $\beta$. In particular, the cluster variables $x[- \alpha_i] , i \in I$   respectively coincide with the cluster variables $x_1, \ldots , x_n$ belonging to the initial cluster previously described.

 \subsection{Representations of affine quantum groups and their $q$-characters} \label{sec : q-characters}

 Let $\mathfrak{g}$ be a simply-laced type simple complex Lie algebra and let $U_q(\widehat{\mathfrak{g}})$ denote the associated quantum affine algebra. The category $\mathcal{C}$ of finite-dimensional representations of $U_q(\widehat{\mathfrak{g}})$ is a monoidal category. It was proved by Chari-Pressley \cite{CP95a} that the irreducible representations in $\mathcal{C}$ are classified (up to isomorphism) by a monoid $\mathcal{P}^+$ whose elements are called dominant monomials and which is defined as the free monoid generated by indeterminates $Y_{i,a} , i \in I , a \in \mathbb{C}^{*}$. Here $I$ denotes an indexing set of the simple roots of $\mathfrak{g}$. We will denote by $L(\mathfrak{m})$ the simple module corresponding to $\mathfrak{m}$ under this classification. Although $\mathcal{C}$ is not symmetric, its Grothendieck ring is commutative. Frenkel-Reshetikhin \cite{FR1} introduced an injective ring homomorphism 
 $$ \chi_q : \enspace K_0(\mathcal{C}) \longrightarrow \mathcal{Y} := \mathbb{Z}[Y_{i,a}^{\pm} , i \in I , a \in \mathbb{C}^*]  $$
 called the $q$-character morphism. The set of Laurent monomials in $\mathcal{Y}$ is endowed with a partial order called the Nakajima order which is given by 
 $$ \mathfrak{m} \leq_{N} \mathfrak{m}' \enspace \Leftrightarrow \enspace \text{$\mathfrak{m}'/ \mathfrak{m}$ is a product of $A_{i,a} , i \in I, a \in \mathbb{C}^{*}$} $$
 where for each $i \in I$ and $a \in \mathbb{C}^{*}$, the Laurent monomial $A_{i,a}$ is defined as 
 \begin{equation}
     \label{eq : def of Ais}
      A_{i,a} := Y_{i,a-1}Y_{i,a+1} \prod_{j \sim i}Y_{i,a}^{-1} . 
 \end{equation}
 For every dominant monomial $\mathfrak{m} \in \mathcal{P}^+$, we have that $\mathfrak{m}$ is the unique greatest element with respect to $\leq_N$ in the set of all Laurent monomials  appearing in $\chi_q(L(\mathfrak{m}))$. In other words, the renormalized $q$-character of $M$ defined as $\mathfrak{m}^{-1} \chi_q(M)$ is a polynomial in the $A_{i,a}^{-1} , i \in I , a \in \mathbb{C}^
 *$ with constant term $1$.

  \subsection{Hernandez-Leclerc categories} \label{sec : HL categories}

 Motivated by exhibiting deep connections between the representation theory of quantum affine algebras and the combinatorics of cluster algebras introduced by Fomin and Zelevinsky \cite{FZ1}, Hernandez-Leclerc defined several monoidal full subcategories of $\mathcal{C}$, which we now briefly recall. 
Let $\IZ$ denote the set of pairs
$$ \IZ := \{ (i,p) , i \in I , p \in i + 2 \mathbb{Z} \} $$
 and let $\YZ$ denote the ring
 $$ \YZ := \mathbb{Z}[Y_{i,p}^{\pm 1} , (i,p) \in \IZ]. $$
Following \cite{HL15}, we denote by $\CZ$ the smallest monoidal  full  subcategory of $\mathcal{C}$ containing the fundamental representations $L(Y_{i,q^p})$ for $p \in i + 2 \mathbb{Z}$. Following the common use, we will be wirting $Y_{i,p}$ as a shorthand notation for $Y_{i,q^p}$ in all what follows. The simple objects of $\CZ$ are the modules $L(\mathfrak{m})$ where $\mathfrak{m}$ is a monomial in the $Y_{i,p}$. Frenkel-Reshetikhin's $q$-character morphism restricts to an injective ring homomorphism 
 $$ \chi_q : \CZ \longrightarrow \YZ . $$
 Using standard notation, we set 
 $$ A_{i,p+1} := Y_{i,p}Y_{i,p+2} \prod_{j \sim i} Y_{j,p+1}^{-1} $$
 for each $(i,p) \in \IZ$. 
  We fix an orientation $Q$ of the Dynkin diagram of $\mathfrak{g}$ and $\xi$ a height function adapted to $Q$. Then we set, for every positive integer $l \geq 1$,
 $$  \Ixi^{(l)} := \{ (i,p) \in \IZ , \enspace \xi(i)-2l \leq p \leq \xi(i)\} \qquad \text{and} \qquad  \Ixi := \{ (i,p) \in \IZ , p \leq \xi(i) \} $$
 and we denote by $\Cxi^{(l)}$ (resp. $\CxiZ$) the smallest full monoidal subcategory of $\CZ$ containing all the fundamental representations $L(Y_{i,p})$ such that $(i,p) \in \Ixi^{(l)}$ (resp. $(i,p) \in \Ixi$). 
Finally, following \cite{HL10,HL16}, we define for each $M \in \CxiZ$ the truncated $q$-character $\tchi_q(M)$ of $M$  as the element of $\YZ$ obtained from $\chi_q(M)$ by dropping all the terms involving variables $Y_{i,p}$ with $p> \xi(i)$. This gives another injective ring homomorphism 
 $$ \tchi_q : \CxiZ \longrightarrow \YxiZ :=\mathbb{Z}[Y_{i,p}^{\pm 1} ,  (i,p) \in \Ixi]  $$
 whose relevance is enlightened by the statement of Theorem~\ref{thm : thm HL/BC for C1} below.

 \subsection{Abelian monoidal categorifications}
  \label{sec : reminders on monoidal categorifications}

   We now provide brief reminders about the general notion of abelian monoidal categorification of a cluster algebra, as well as recollections of the main results involving Hernandez-Leclerc categories. Note that although the usual terminology is simply monoidal categorification, we will refer to this framework as abelian monoidal categorification, in contrast with the construction presented in this paper which deals with triangulated monoidal categories. 
   
   The notion of abelian monoidal categorification of a cluster algebra was introduced by Hernandez-Leclerc \cite{HL10}. Given a cluster algebra $\mathcal{A}$, an abelian monoidal categorification of $\mathcal{A}$ consists in an abelian monoidal category $ \mathcal{C}$ together with a ring isomorphism $\mathcal{A} \simeq K_0(\mathcal{C})$ sending cluster monomials in $\mathcal{A}$ to isomorphism classes of simple objects in $\mathcal{C}$ (these objects are then necessarily real, i.e. their tensor square is simple). The simplicity of tensor products of objects categorifying cluster variables or monomials is therefore deeply related to the cluster structure of $\mathcal{A}$. For instance, the simple objects categorifying frozen cluster variables (if there are any) in $\mathcal{A}$ satisfy the property that their tensor product with any object categorifying a cluster monomial is simple.

   The case of the category $\Cxi^{(1)}$ will be of particular interest to us in this paper, in view of the following result, which was established  by Hernandez-Leclerc \cite{HL10} in the case where $Q$ is bipartite of type $A_n$ or $D_4$, and generalized by Brito-Chari \cite{BC} to all orientations in type $A_n$. It follows from the results by Kashiwara-Kim-Oh-Park in the general case.

  \begin{thm}[Hernandez-Leclerc \cite{HL10}, Brito-Chari \cite{BC}, Kashiwara-Kim-Oh-Park \cite{KKOP2}] 
  \label{thm : thm HL/BC for C1}

   Let $Q$ be any Dynkin quiver and $\xi$ be a height function adapted to $Q$. Then the category $\Cxi^{(1)}$ provides a monoidal categorification of the finite type cluster algebra $\AQ$. Moreover, for each almost positive root $\beta$, the truncated $q$-character of the corresponding real simple module $L[\beta]$ in $\Cxi^{(1)}$ coincides (up to a torus isomorphism) with the cluster expansion of $x[\beta]$ with respect to the initial cluster of $\AQ$.
    \end{thm}

 \begin{rk}
  It was later proved by Hernandez-Leclerc that the Grothendieck ring of $\CZ^{(l)}$ (resp. $\CxiZ$) carries a cluster algebra structure with an initial cluster given by the simple modules (which are Kirillov-Reshetikhin modules) given by $L(Y_{i, \xi(i)} \cdots Y_{i, \xi(i)-2k})$ with $0\leq k \leq l$ (resp. $k \geq 0$). Kashiwara-Kim-Oh-Park \cite{KKOP1,KKOP2} then proved that the categories $\CxiZ$ and $\Cxi^{(l)}$ actually provide  monoidal categorifications of these cluster structures. 
 \end{rk}

 \begin{rk} \label{rk : frozens in abelian monoidal categorifications}
 In the case of the category $\Cxi^{(1)}$, the initial cluster of $\AQ$ is categorified by the following real simple modules: with the notations of Section~\ref{sec : reminders AQ}, we have that for every $i \in I$ the (unfrozen) initial variable $x_i$ is categorified by the fundamental representation $L(Y_{i, \xi(i)})$ while the frozen variable $X_i$ is categorified by the Kirillov-Reshetikhin modules $L(Y_{i, \xi(i)-2}Y_{i, \xi(i)})$. Moreover in that case every simple module in $\Cxi^{(1)}$ categorifies a cluster monomial (and hence is real). In particular we have that $ L(Y_{i, \xi(i)-2}Y_{i, \xi(i)}) \otimes M $ is simple for any simple module $M$ in $\Cxi^{(1)}$. 
   \end{rk}
  
   \subsection{Invariants associated to pairs of simple modules}

  It will be crucial to us later in this paper to show the existence of certain short exact sequences in $\CZ$ (or more specifically in $\Cxi^{(1)}$). A powerful tool for such purpose has been introduced by Kashiwara-Kim-Oh-Park \cite{KKOP1,KKOP2}, which can be viewed as the pendant in the setup of quantum affine algebras of an analogous tool developed by Kang-Kashiwara-Kim-Oh \cite{KKKO} for quiver Hecke algebras. Namely, for each pair $M,N$ of simple modules in $\CZ$, Kashiwara-Kim-Oh-Park defined a positive integer $\bdelta(M,N)$ out of the data of the poles of the $R$-matrices $R_{M,N} : M \otimes N \rightarrow N \otimes M$ and $R_{N,M} : N \otimes M \rightarrow M \otimes N$ in $\CZ$. Among the numerous interesting properties of these invariants, we will be using the following. 

   \begin{thm}[Kashiwara-Kim-Oh-Park] \label{thm : KKOP}
    \begin{enumerate}[(i)]
    \item (cf. {{\cite[Corollary 3.17]{KKOPCompos20}}}  Let $M$ and $N$  be simple modules in $\CZ$ such that at least one of them is real, and assume $\bdelta(M,N)=0$. Then the tensor products $M \otimes N$ and $N \otimes M$ are simple and are isomorphic to each other. 
    \item (cf. {{\cite[Proposition 4.7 and Corollary 4.12]{KKOPCompos20}}}   Let $M$ and $N$ be simple modules in $\CZ$ such that at least one of them is real, and assume $\bdelta(M,N)=1$. Then $M \otimes N$ is of length $2$. Moreover each of its simple components strongly commute with both $M$ and $N$. 
    \item (cf. {{\cite[Proposition 4.2]{KKOPCompos20}}}) Let $M,N,L$ be real simple modules in $\CZ$. Then we have  $\bdelta( \mathrm{hd}(M \otimes N),L) \leq  \bdelta(M,L) + \bdelta(N,L)$.
    \item (cf. {{\cite[Lemma 3.10 and Corollary 3.19]{KKOPCompos20}}}) Let $n \geq 1$ and let $M$ and $L$ be real simple modules in $\CZ$. Then $\bdelta(M^{\otimes n},L) = n \bdelta(M,L)$.
    \item (cf.  {{\cite[Lemma 3.4]{KKOPJEMS24}}}) Let $M$ be a simple module in $\CZ$ and let $L$ be a fundamental representation in $\CZ$. Then $\bdelta(\mathrm{hd}(M \otimes L),L) = \max(0 , \bdelta(M,L)-1)$. 
    \end{enumerate}
    
   \end{thm}

    We will also need the following result due to R. Fujita.

     \begin{thm}[Fujita  \cite{Fujita}] \label{thm : Fujita}
     Let $(i,p)$ and $(j,s)$ in $\IZ$. Then the tensor product  $L(Y_{i,p}) \otimes L(Y_{j,s})$  is simple if and only if $\mathrm{Ext}^1(M_{i,p},M_{j,s}) = \mathrm{Ext}^1(M_{j,s},M_{i,p}) = 0$ in $\DbRepQ$. 
         
     \end{thm}

\section{The underlying category}
\label{sec : the underlying categ}


  \subsection{Quasi-additive functions}

  Extending the  terminology from  Section~\ref{sec : additive functions}, we consider the following  class of integer-valued functions on $\ZQz$.

   \begin{deftn} \label{def : quasi-additive functions}
     A function $f : \ZQz \longrightarrow \mathbb{Z}$ will be called quasi-additive if $\widetilde{f}$ has finite support, i.e. $\widetilde{f}(x)=0$ on all but finitely many vertices of $\ZQ$.
   \end{deftn}
For each $x \in \ZQz$, we consider the unique  quasi-additive function $h_x$ determined by the requirements 
    \begin{enumerate}
        \item $h_x(y)= \dim \mathrm{Hom}_{\DbRepQ}(x,y)$  for every vertex $y$ on the section $Q_x$ of $\ZQ$,
        \item $ \widetilde{h_x} = \delta_x$,
    \end{enumerate}
   where $\delta_x$ denotes the function given by $\delta_x(y) := 1$ if $x=y$ and $0$ otherwise. For a given $x \in \ZQz$, the function $h_x$ should be viewed as a semi-infinite version of the hammock functions from Section~\ref{sec : additive functions}, as illustrated below. 

\begin{example} \label{ex : quasi-additive function}
Consider the Dynkin quiver $Q$ of type $A_4$ given by $1 \rightarrow 2 \rightarrow 3 \leftarrow 4$ as in Example~\ref{ex : additive function}. We picture below the function $h_x$ where $x$ is the vertex of $\ZQ$ outlined in red. 

 $$ \xymatrixrowsep{1pc} \xymatrixcolsep{1pc} 
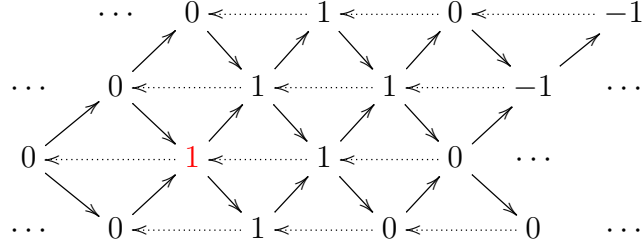
\xymatrix{  
   {} &  \cdots   & 0 \ar[rd]   &    {}   &    1  \ar@{.>}[ll] \ar[rd]   & {}  & 0 \ar@{.>}[ll] \ar[rd] & {} & -1 \ar@{.>}[ll]   \\
    \cdots    &  0 \ar[ru] \ar[rd]   &    {}      &  1 \ar@{.>}[ll]  \ar[ru]  \ar[rd]  &{}  &  1 \ar@{.>}[ll] \ar[ru]  \ar[rd] & {} & -1 \ar@{.>}[ll]  \ar[ru] & \cdots \\
      0 \ar[ru] \ar[rd]  &     {}     &    \textcolor{red}{1} \ar[ru]  \ar@{.>}[ll]   \ar[rd]   & {} & 1 \ar@{.>}[ll]  \ar[rd] \ar[ru] & {} & 0 \ar[ru] \ar@{.>}[ll] \ar[rd] & \cdots  \\
    \cdots   &   0  \ar[ru]   &  {}       &   1 \ar@{.>}[ll]  \ar[ru]   & {} & 0  \ar@{.>}[ll]   \ar[ru]& {} & 0 \ar@{.>}[ll] & \cdots 
    } $$

\end{example}
   
 We now prove a couple of important  properties of these functions.

  \begin{lem} \label{lem : linear independece of hammock functions}
      The functions $h_x , x \in \ZQz$ are linearly independent.
  \end{lem}

 \begin{proof}
     Assume $r \geq 1$ and  $a_1, \ldots , a_r \in \mathbb{Z}$  and $x_1, \ldots , x_r \in \ZQz$ are such that $a_1 h_{x_1} + \cdots + a_r h_{x_r} = 0$. We can choose $i$ such that $x_i$ is to the left of the other $x_j$. Wlog $i=1$. Hence we have $h_{x_k}(x_1) = 0$ for all $k>1$ and $h_{x_1}(x_1) = 1$ by definition of the functions $h_x$. This implies $a_1=0$. The lemma thus follows by a straightforward induction. 
  \end{proof}

  \begin{lem} \label{lem : mutation for functions}
   For every $x \in \ZQz$ we have that 
   $$ h_x + h_{\tau^{-1}x} = \delta_x + \sum_{x \rightarrow y} h_y .  $$
  \end{lem}

   \begin{proof}
     Let us fix $x \in \ZQz$. It is then straightforward to check that 
     $$ \widetilde{\delta_x} = \delta_x + \delta_{\tau^{-1}x} - \sum_{x \rightarrow y} \delta_y  . $$
    It follows from this that the function 
    \begin{equation} \label{eq : aux function}
         h_x + h_{\tau^{-1}x} - \sum_{x \rightarrow y} h_y - \delta_x
     \end{equation}
     is additive and hence is entirely determined by its values on any section of $\ZQ$. Let us consider the section starting at $x$. For any $z$ on that section, we have that $\dim_{\mathbf{k}} \mathrm{Hom}(x,z) =1$  so that $h_x(z) =1$ and $\mathrm{Ext}^1(z,x) = 0 $ which implies that $h_{\tau^{-1}x}(z) = 0$ using the Auslander-Reiten formula. On the other hand, $z$ belongs to exactly one of the sections starting at each $y$ such that there is an arrow $x \rightarrow y$ in $\ZQ$, unless $z$ is $x$ itself. In other words we have that $ \sum_{x \rightarrow y}h_y (z) = 1 $ if $z \neq x$, and $0$ if $z=x$. Putting all of this together we see that the function~\eqref{eq : aux function} vanishes on all the vertices of the section starting at $x$. The lemma is proved. 
   \end{proof}

    \subsection{Two partial orders on quasi-additive functions}

     The set of all quasi-additive functions is endowed with an obvious partial order, namely 
     $$ f \leq g \enspace \Leftrightarrow \enspace \forall x \in \ZQz,  f(x) \leq g(x) .  $$
     However in what follows we will need a more refined partial order. We introduce the following covering relation:
     $$ f \lessdot g \quad  \Leftrightarrow \quad  \exists x \in \ZQz, \enspace g-f = \delta_x \enspace \text{and} \enspace \widetilde{g}(x)>0 .   $$
     We then define a partial order $\preccurlyeq$ on the set of all quasi-additive functions on $\ZQz$ as the transitive closure of $\lessdot$. In particular we have that $f \preccurlyeq g \Rightarrow f \leq g$.

 \subsection{The category $\MQ$}
  \label{sec : category MQ}
 
 We denote by $\IQ$ the set of all isomorphism classes of indecomposable objects in $\DbRepQ$. Recall that a multiset of elements of $\IQ$ is a collection $\{x_k , k \in K\}$ where $K$ is a set and $x_k$ is an indecomposable object in $\DQ$ for each $k \in K$ (equivalently, it can be viewed as a  function from $K$ to $\IQ$). Note that such a function shall usually not be injective; in other words, any given indecomposable object  in $\DbRepQ$ may appear several times (up to isomorphism) in the same collection (but only finitely many times !). In this work, we will exclusively be considering at most countable  multisets, i.e.  functions from a finite or  countable set $K$ to $\IQ$. 
  Given two multisets $X$ and $Y$ we can define their intersection $X \cap Y$ as the largest multiset contained in both $X$ and $Y$. 
 
 We define a $\mathbf{k}$-category $\tMQ$ as follows. The class of objects of $\tMQ$ consists of all pairs $(X,h)$ where $X$ is an at most countable multiset of elements of $\IQ$ and $h$ is a quasi-additive function on $\ZQz$. We then define the notion of Serre tilting of such pairs as follows. Recall that $S$ denotes the Serre functor $S = \tau \Sigma = \Sigma \tau$ in the category $\DbRepQ$. 
 
   \begin{deftn} \label{def : Serre tiltings}
 Given any indecomposable object $(X,h)$ in $\tMQ$ and any \textit{finite} subset $Y = \{ y_1 , \ldots  , y_r\} \subseteq X$, we define the \textbf{Serre tilting} of $(X,h)$ at $Y$ as the object of $\tMQ$ given by  
 $$ \mu_Y(X,h) := \left(  (X \setminus Y) \sqcup SY \enspace , \enspace  h - \sum_{y \in Y} \delta_y \right) $$
 where $SY$ stands for $SY := \{S y_1, \ldots , S y_r\}$.  For simplicity we will write $\mu_y(X)$ for $\mu_{\{y \}}(X)$.
   \end{deftn}
  
   We now define the class of morphisms of the category $\tMQ$. 
 Given two at most countable sets $K$ and $L$, we denote by $\mathrm{Bij}(K,L)$ the set of all bijections from $K$ to $L$. In particular $\mathrm{Bij}(K,L) = \emptyset$ if $K$ and $L$ are finite of different cardinalities. Moreover for any function $\sigma : K \rightarrow L$, we set 
  $$  \mathrm{Supp}_{X,Y}(\sigma) := \{k \in K \mid y_{\sigma(k)} \neq x_k \} .  $$
  Given two multisets $X = \{x_k, k \in  K\}$ and $Y = \{y_l , l \in L\}$ and two  quasi-additive functions $f$ and $g$ on $\ZQz$, we then define the morphism space from $(X,f)$ to $(Y,g)$ in $\widetilde{\MQ}$ as follows:
   \begin{equation} \label{eq : def of morphisms in CQ}
  \begin{split} 
      \mathrm{Hom}_{\tMQ} \left( (X,f),(Y,g) \right)  := 
       \bigoplus_{  \substack{ \sigma \in \mathrm{Bij}(K,L) \\  \sharp \mathrm{Supp}_{X,Y}(\sigma) < \infty }} & \bigotimes_{k \in \mathrm{Supp}_{X,Y}(\sigma)} \mathrm{Hom}_{\DQ}(x_k , y_{\sigma(k)})  \\
       \qquad \qquad \qquad \qquad \qquad \qquad  \qquad \qquad  & \text{if $(Y,g) = \mu_Z(X,f), \enspace Z \subseteq Y$}
      \\
      {} \\
     \mathrm{Hom}_{\tMQ} \left( (X,f),(Y,g) \right)  &:= \enspace 0 \qquad  \text{otherwise.}
  \end{split}
  \end{equation}
   Note that each summand on the right hand side is a tensor product of finitely many $\mathbf{k}$-vector spaces (as $\DbRepQ$ is a $\mathbf{k}$-category) so that $\tMQ$ is a $\mathbf{k}$-category.

    \begin{rk}
        In most of what follows we will be considering only the case where $X$ and $Y$ are finite with same cardinality so that
   $$ \mathrm{Hom}_{\tMQ} \left( (X,f),(Y,g) \right) := \bigoplus_{ \sigma \in \mathrm{Bij}(K,L)} \bigotimes_{k \in K} \mathrm{Hom}_{\DQ}(x_k , y_{\sigma(k)}) $$
   if $(Y,g)$ is a Serre tilting of $(X,f)$.
    \end{rk}
   
   Using standard procedures,  we can then make it into an additive (and hence $\mathbf{k}$-linear) category $\MQ$. We claim that the objects of $\tMQ$ will be indecomposable in $\MQ$. Let us begin with an elementary lemma.

  \begin{lem} \label{lem : both ways morphisms}
      Let $(X,f)$ and $(Y,g)$ be two objects in $\tMQ$. Then we have that $\mathrm{Hom}_{\MQ}((X,f),(Y,g))$ and $\mathrm{Hom}_{\MQ}((Y,g),(X,f))$ are both non zero if and only if $(X,f) \simeq (Y,g)$ in $\tMQ$. 
  \end{lem}

 \begin{proof}
     We will write the proof in the case where $X$ and $Y$ are finite. Assume 
     $$ \mathrm{Hom}_{\MQ}((X,f),(Y,g)) \neq 0  \quad \text{and}  \quad  \mathrm{Hom}_{\MQ}((Y,g),(X,f)) \neq 0 . $$
     Firstly, this clearly implies $f=g$. It moreover implies that  $X$ and $Y$ have same cardinalities so we may write $X := \{x_1, \ldots , x_m\}$ and $Y := \{y_1, \ldots , y_m\}$, with $m \in \mathbb{Z}_{\geq 1}$. We can choose $j \in \{1, \ldots , m\}$ such that for all  $i$ is such that $x_i$ is not isomorphic to $x_j$ and all indecomposable objects $x$ in $\DbRepQ$ such that  $\mathrm{Hom}(x,x_i) \neq 0$ we have $\mathrm{Hom}_{\DbRepQ}(x_j,x) =0$. Without loss of generality we assume $j=1$. As $\mathrm{Hom}_{\MQ}(X,Y) \neq 0$ there is $k \in \{1, \ldots, m\}$ such that $\mathrm{Hom}_{\DbRepQ}(x_1,y_k) \neq 0$. As $\mathrm{Hom}_{\MQ}(Y,X) \neq 0$ there is $l \in \{1, \ldots , m \}$ such that $\mathrm{Hom}_{\DbRepQ}(y_k ,x_l) \neq 0$. Thus we must have that $x_l \simeq x_1$ in $\DbRepQ$ and hence $y_k \simeq x_1$ as well. Moreover we get the existence of nonzero morphisms respectively from $X \setminus \{x_1\}$ to $Y \setminus \{y_k\}$ and from $Y \setminus \{y_k\}$ to $X \setminus \{x_l\}$. But the latter is isomorphic to $X \setminus \{x_1\}$ as $x_1 \simeq x_l$ in $\DbRepQ$ and hence we get that both $\mathrm{Hom}_{\MQ}(X',Y')$ and $\mathrm{Hom}_{\MQ}(Y',X')$ are non zero, where $X' := X \setminus \{x_1\}$ and $Y' := Y \setminus \{y_k\}$. We can then repeat the same argument inductively to conclude that $X \simeq Y$ as multisets and thus $(X,f) \simeq (Y,g)$ in $\tMQ$. 
 \end{proof}

  The claim follows as explained below:

   \begin{cor} \label{cor : indecomposable objects in MQ}
       Let $(X,f)$ be any object in $\tMQ$. Then $(X,f)$ is indecomposable in $\MQ$. 
   \end{cor}

    \begin{proof}
        Assume $(X,f)$ can be decomposed as $(X,f) \simeq (Y,g) \oplus (Z,h)$ in $\MQ$ with $(Y,g)$ and $(Z,h)$ non-zero objects in $\tMQ$. Then in particular we have that $\mathrm{Hom}_{\tMQ}((X,f),(Y,g))$ and $\mathrm{Hom}_{\tMQ}((Y,g),(X,f))$ are both non zero (and similarly for $(Z,h)$) which implies by Lemma~\ref{lem : both ways morphisms}  that $(Y,g) \simeq (X,f)$ in $\tMQ$ (and hence in $\MQ$), and for the same reason $(X,f) \simeq (Z,h)$ in $\MQ$. Thus we obtain that $(X,f) \simeq (X,f) \oplus (X,f)$ in $\MQ$.  In the case where $Q$ is a Dynkin quiver, it follows from the  construction of $\tMQ$ that  $\mathrm{Hom}_{\MQ}((X,f),(X,f))$ is a finite-dimensional $\mathbb{C}$-vector space. The fact that $\mathbb{C}$ has characteristic zero gives the contradiction, which shows that $(X,f)$ is indecomposable in $\MQ$. 
    \end{proof}

   Given two indecomposable objects $(X,f)$ and $(Y,g)$ in $\MQ$ with $X = \{x_k , k \in K\}$ and $Y = \{y_l , l \in L\}$, we set 
\begin{equation} \label{eq : monoidal structure on CQ}
 (X,f) \otimes (Y,g) :=  \left( X \sqcup Y , f + g \right) .  
\end{equation}   
   where $X \sqcup Y$ is the multiset defined by
   $$ X \sqcup Y := \{z_p , p \in K \sqcup L \} , \enspace z_p := \begin{cases}
       x_p & \text{if $p \in K$,} \\
       y_p & \text{if $p \in L$. }
   \end{cases} 
   $$
Moreover, given two non-trivial morphisms $\varphi_1 : (X_1,f_1) \rightarrow (Y_1,g_1) $ and $\varphi_2 : (X_2,f_2) \rightarrow (Y_2,g_2)$ in $\MQ$, we  get a canonical non-trivial morphism $\varphi_1 \otimes \varphi_2 : (X_1 , f_1) \otimes (X_2,f_2) \rightarrow (Y_1,g_1) \otimes (Y_2,g_2)$ in $\MQ$ as $(Y_1 \sqcup Y_2, g_1 + g_2) = \mu_{Z_1 \sqcup Z_2} (X_1 \sqcup X_2 , f_1 + f_2)$ where $Z_1$ and $Z_2$ are such that $(Y_i,g_i) = \mu_{Z_i}(X_i,f_i)$ for each $i \in \{1,2\}$. It is straightforward to check that this construction is natural and hence~\eqref{eq : monoidal structure on CQ} endows $\MQ$ with a symmetric monoidal structure.

\subsection{The category $\HQ$} \label{sec : category CQ and KQ}

 We now introduce the category $\HQ$ which will serve as underlying category for all the constructions presented in the next sections. We begin by defining certain distinguished objects (that we will refer to as  hammock objects) in $\MQ$ that will be playing a crucial role in what follows.  For every  $x$ in $\IQ$, we define the $\textbf{hammock multiset}$ associated to $x$  as the multiset $H(x)$ of elements of $\IQ$ given by
  \begin{equation} \label{eq : def of hammock multisets}
       H(x) := \bigsqcup_{y \in \IQ} \bigsqcup_{m=1}^{\dim_{\mathbf{k}} \mathrm{Hom}_{\DQ}(x,y)} \{y  \} .     
  \end{equation}
  
  In other words, $H(x)$ is the multiset containing $m_x(y)$ copies of each $y \in \IQ$, where $m_x(y) := \dim_{\mathbf{k}} \mathrm{Hom}_{\DQ} (x,y)$ for each $y \in \IQ$. Moreover, for every $x \in \IQ$, we define the \textbf{hammock object} associated to $x$ as the indecomposable object $Y(x)$ of $\MQ$ given by
    \begin{equation} \label{eq : def of hammock objects}
         Y(x) := \left(  H(x) , h_x \right) .   
    \end{equation}
   We will also need the following (indecomposable) objects in $\MQ$:
\begin{equation} \label{eq : def of F objects}
     \forall x \in \IQ, \enspace F(x) :=  \left( \{ Sx , \Sigma x \} , 0 \right) . 
\end{equation}
We then consider  the following collections of indecomposable objects in $\MQ$:
  \begin{equation} \label{eq : B(x)}
      \forall x \in \IQ, \enspace \mathcal{B}(x) := \{ (X,h) \mid h \preccurlyeq h_x \enspace \text{and} \enspace  \mathrm{Hom}_{\MQ} \left( Y(x) , (X,h) \right) \neq 0 \} . 
  \end{equation}
We then define the category $\HQ$ as the smallest additive monoidal full subcategory of $\MQ$ containing $F(x)$ as well as all the objects belonging to $\mathcal{B}(x)$, for all $x \in \IQ$. 
The pairs of the form $(X_1,f_1) \otimes \cdots \otimes (X_r,f_r)$ where for each $k$, $(X_k,f_k)$ is of the form $Y(x_k)$  for some $x_k \in \ZQz$ will be called \textit{dominant objects}. The set of all dominant objects will be denoted by $\ObQ^+$.

 \begin{rk} \label{rk on dominant objects}
    It follows from Lemma~\ref{lem : linear independece of hammock functions} that a dominant object is entirely determined (up to isomorphism) by its quasi-additive function.  
 \end{rk}

  The next statement shows how the Serre tilting of a hammock object is related to other hammock objects. 

 \begin{lem} \label{lem : mutation for objects in CQ}
     For each $x \in \IQ$ we have that 
     $$ \mu_xY(x) \otimes Y(\tau^{-1}x) \simeq  F(x) \otimes \bigotimes_{x \rightarrow y} Y(y) . $$
 \end{lem}

  \begin{proof}
  In view of Lemma~\ref{lem : mutation for functions}, it suffices to show the following identity of multisets 
  $$ H(x) \sqcup H(\tau^{-1}x) = \{x , \Sigma x \} \sqcup \bigsqcup_{x \rightarrow y} H(y) .  $$
  This follows from the fact that $x \rightarrow \bigoplus_{x \rightarrow y} y \rightarrow \tau^{-1}x \rightarrow \Sigma x$ is an Auslander-Reiten triangle in $\DbRepQ$ and that $\DbRepQ$ has weak kernels. 
  \end{proof}

The following observation is immediate from the definition of $\MQ$:

 \begin{lem}
     For every $x \in \IQ$ we have that $\mu_xY(x)$ belongs to $\HQ$. 
 \end{lem}



 \section{The subcategory $\HQ^{(1)}$}
 \label{sec : restriction to the section}

In this section, we introduce a certain monoidal subcategory  (which will not be a full subcategory) of $\HQ$, denoted $\HQ^{(1)}$. It will serve as underlying category for all the constructions and results in the following sections. In particular, its bounded homotopy category will eventually be the triangulated monoidal category relevant for categorifying the cluster structure associated to $Q$, when $Q$ is a Dynkin quiver.

 \subsection{The subcategory $\HQ^{(1)}$}
 \label{sec : the subcategory CQone}

We also fix a height function $\xi : I \rightarrow \mathbb{Z}$ adapted to $Q$, i.e. $\xi(j) = \xi(i)-1$ for every arrow $i \rightarrow j$ in $Q$. We then denote by $x_i$ the vertex of $\ZQ$ given by $x_i := (i,\xi(i))$ for every $i \in I$.
 In order to define the morphisms in the subcategory $\HQ^{(1)}$, we will need the following observation. 

    \begin{lem} \label{lem : intersection of hammocks}
        There exists an indecomposable object $x$ in $\DbRepQ$ such that 
        $$ \forall i \in I,  \enspace \mathrm{Hom}_{\DbRepQ}(\tau x_i , x) \neq 0  . $$ 
    \end{lem}

     \begin{proof}
          let us consider an abelian subcategory $\mathcal{C}$ of $\DbRepQ$ equivalent to $\mathrm{Rep}Q$ and whose indecomposable projective objects  are the $\tau x_i , i \in I$. We then have that for any indecomposable objects $x$ and $y$ in $\mathcal{C}$, 
  $$ \dim_{\mathbf{k}} \mathrm{Hom}_{\mathcal{C}}(x,y) = \langle \boldsymbol{\dim} x , \boldsymbol{\dim} y \rangle_Q $$
  where $\langle  \cdot ,  \cdot \rangle_Q$ denotes the Euler-Ringel form of $Q$. Consider $x$ an indecomposable object in $\mathcal{C}$ whose dimension vector is given by $\sum_{i \in I} \alpha_i$. Such an object always exists if $Q$ is of Dynkin or tame type. We then have that 
  \begin{align*}
 \langle \boldsymbol{\dim} \tau x_i , \boldsymbol{\dim} x \rangle_Q &=  \langle \sum_{ \substack{j \in I \\ i \rightsquigarrow j}} \alpha_i , \sum_{k \in I} \alpha_k \rangle_Q  = 1 + \sum_{ \substack{  i \rightsquigarrow j \\ k \rightsquigarrow i ,k \neq i}} \langle \alpha_j , \alpha_k \rangle_Q \\
 &= 1 + \sum_{ k \rightarrow i} \langle \alpha_i , \alpha_k \rangle_Q \geq 1 . 
  \end{align*}
   This shows that $\dim_{\mathbf{k}} \mathrm{Hom}_{\DbRepQ} (\tau x_i  , x) \neq 0$ for all $i \in I$ and hence  $x \in \bigcap_{i \in I} H(\tau x_i)$ which establishes the lemma. 
     \end{proof}

 From now we fix an indecomposable object $x$ in $\DbRepQ$ as given by Lemma~\ref{lem : intersection of hammocks}. By Lemma~\ref{lem : intersection of hammocks}, we can choose and fix once for all a non trivial morphism $f_i \in \mathrm{Hom}_{\DbRepQ}(\tau x_i,x)$ for each $i \in I$. We denote by $\overline{f_i} \in \mathrm{Hom}_{\DbRepQ}(x , S \tau x_i)$ the morphism given by the isomorphism~\eqref{eq : isom def of Serre functor}. 

  We now consider the following collections of indecomposable objects in $\HQ$:
 $$ \forall i \in I, \enspace \mathcal{B}_i := \mathcal{B}(\tau x_i) \cap \left\{ \mu_ZY(\tau x_i) \mid Z \subseteq \{\tau x_i , i \in I\} \right\} . $$
 The category $\HQ^{(1)}$ is then defined as the following subcategory of $\HQ$. The class of objects of $\HQ^{(1)}$ consists in the smallest family of objects of $\HQ$ stable under both the additive structure and the monoidal structure of $\HQ$, and containing $F(\tau x_i)$ and $ Y(x_i)$ as well as all the objects belonging to $\mathcal{B}_i$, for every $i \in I$. Given two indecomposable objects $(X,f)$ and $(Y,g)$ in $\HQ^{(1)}$, the set of morphisms from $(X,f)$ to $(Y,g)$ in $\HQ^{(1)}$ is given as follows. By construction of $\HQ$, we can write $(Y,g) = \mu_Z(X,f)$ for some finite subset $Z \subseteq X$. Then $\mathrm{Hom}_{\HQ^{(1)}}((X,f),(Y,g))$ is defined as the subvector space of $\mathrm{Hom}_{\HQ}((X,f),(Y,g))$ generated by the morphisms of the form $f_1 \otimes \cdots \otimes f_r$ with $f_k \in \mathrm{Hom}_{\DbRepQ}(x_k,y_{\sigma(k)})$  such that $y_{\sigma(k)} \in \{x\} \cup  \{S \tau x_i , i \in I\}$ for each $k$ such that $x_k \in Z$. 
 We then denote by $\KQ^{(1)}$ the bounded homotopy category of $\HQ^{(1)}$, which is a triangulated monoidal category. In all what follows we will simply write $\mu_iX$ instead of $\mu_{\tau x_i}X$ for any (indecomposable) object $X$ in $\HQ^{(1)}$ and any $i \in I$.


 \subsection{Weak cokernels in $\HQ^{(1)}$}

 This section is devoted to the proof of Theorem~\ref{thm : existence of morphisms eta} below, who states the existence in $\HQ^{(1)}$ of a distinguished family of  morphisms in  that will play a crucial role in what follows. Before stating the theorem, let us introduce a piece of notation: for any indecomposable object $(X,h)$ in $\HQ^{(1)}$, we  set
 $$ \Sigma(X,h) := \{ i \in I \mid \tau x_i \in X  \enspace \text{and} \enspace \widetilde{h}(\tau x_i) >0 \} .  $$

  \begin{thm} \label{thm : existence of morphisms eta}
    There exists a family of irreducible morphisms 
    $$\{\eta_i^{(X,h)}   \in \mathrm{Hom}_{\HQ^{(1)}} \left( (X,h) , \mu_i (X,h) \right) , \enspace \text{$(X,h)$ indecomposable in $\HQ^{(1)}$}, i \in \Sigma(X,h) \}$$ 
    satisfying the following properties:
     \begin{enumerate}
         \item For every $i \in \Sigma(X,h)$ and every $j \in I$ such that  $j \rightsquigarrow i$ in $Q$, we have that  
         $$j \in \Sigma( \mu_i(X,h)) \qquad \text{and} \qquad  \eta_j^{\mu_i(X,h)} \circ \eta_i^{(X,h)} = 0 . $$ 
         \item For any object $T$ in $\HQ^{(1)}$, and any morphism $f \in \mathrm{Hom}_{\HQ^{(1)}}(\mu_i(X,h), T)$, we have that 
         $$ f \circ \eta_i^{(X,h)} = 0 \quad \Rightarrow \quad \text{$f$ factors through $\bigoplus_{ \substack{j \rightsquigarrow i \\ j \in \Sigma(\mu_i(X,h))}} \eta_j^{\mu_i(X,h)}$.} $$
     \end{enumerate}
  \end{thm}

 Theorem~\ref{thm : existence of morphisms eta} will be proved by constructing the morphisms $\eta_i^{(X,h)}$ explicitly for each indecomposable object $(X,h)$ in $\HQ^{(1)}$ and each $i \in \Sigma(X,h)$. Before providing the definition of these morphisms, we  need a couple of preliminary lemmas.

 \begin{lem} \label{lem : positivity on the section}
     For every indecomposable object $(X,h)$ in $\HQ^{(1)}$, we have that $\widetilde{h}(\tau x_i) \geq 0$ for all $i \in I$. 
 \end{lem}

  \begin{proof}
    The statement is obviously true if $(X,h)$ is dominant as in this case $h$ is (by definition) a sum of $h_{\tau x_k}$, and $\widetilde{h_{\tau x_k}}(z) = \bdelta_{\tau x_k, z} \geq 0$ for any $z \in \ZQz$. Assume the statement is true for $(X,h)$ (indecomposable) in $\HQ^{(1)}$. Then for each $j \in I$ such that $\widetilde{h}(\tau x_j)>0$ we have that 
     $$ \forall i \in I, \enspace  \widetilde{h -\bdelta_{\tau x_j}} (\tau x_i) = \widetilde{h}(\tau x_i) - \begin{cases}
         1 & \text{if $i=j$,} \\
         -1 & \text{if $j \rightarrow i$,} \\
         0 & \text{otherwise,}
     \end{cases} $$
     which implies that $\widetilde{h -\bdelta_{\tau x_j}} (\tau x_i) \geq 0$ for all $i \in I$ and hence the statement is true for any Serre tilting  $ \mu_j(X,h)$ in the category $\HQ^{(1)}$. By induction, this establishes the lemma.
  \end{proof}

Before moving on, we need one more piece of notation. For every $i \in I$, define 
 $$ H_i := H(\tau x_i) \cap \{\tau x_k , k \in I\} = \{ \tau x_j \mid j \rightsquigarrow i\}.$$
 where the second equality is due to the fact that $\mathrm{Hom}_{\DbRepQ}(\tau x_j , \tau x_i) =1$ if and only if there is an oriented path from $i$ to $j$ in $Q$, and is zero otherwise. 

\begin{lem} \label{lem : preliminary for def of eta}
Let $(X,h)$ be an indecomposable object in $\HQ^{(1)}$. Then we have that 
$$ \bigsqcup_{i \in I} H_i^{\sqcup \widetilde{h}(\tau x_i)} \subset X . $$
\end{lem}

 Note that the multiset $ H_i^{\sqcup \widetilde{h}(\tau x_i)}$ is well-defined thanks to Lemma~\ref{lem : positivity on the section}. 

 \begin{proof}
    Assume $(X,h) = Y(\tau x_i)$ with $i \in I$. Then by definition we have that $X = H(\tau x_i) \supset \tHi$  and $\widetilde{h}(\tau x_i) = 1$ so the statement holds  in that case, and hence holds for every  (indecomposable) dominant object. Assume the statement has been proved for some indecomposable object $(X,h)$. We shall prove that the statement still holds for  $(X',h') := \mu_j(X,h)$ for each $j \in I$ is such that $\widetilde{h}(\tau x_j)>0$. Recalling Definition~\ref{def : Serre tiltings} we have that 
   $$ \bigsqcup_{i \in I} H_i^{\sqcup \widetilde{h'}(\tau x_i)} = H_j^{\sqcup \widetilde{h}(\tau x_j)-1} \sqcup \bigsqcup_{k \rightarrow j}H_k^{\sqcup \widetilde{h}(\tau x_k)+1} \sqcup \bigsqcup_{l}H_l^{\widetilde{h}(\tau x_l)} $$
   where $l$ runs over all elements of $I$ that are neither $j$ itself nor satisfy $l \rightarrow j$. Now, it is straightforward to check that 
$$ H_j = \{ \tau x_j \} \sqcup \bigsqcup_{k \rightarrow j} H_k $$
because, as the underlying graph of $Q$ is assumed to be a tree, a vertex $i$ such that $i \rightsquigarrow j$ can satisfy $i \rightsquigarrow k$ for only a unique $k \rightarrow j$ (except if $i$ is $j$ itself). 
Therefore we get 
$$ \{ \tau x_j \} \sqcup \bigsqcup_{i \in I} H_i^{\sqcup \widetilde{h'}(\tau x_i)} = \bigsqcup_{i \in I} H_i^{\sqcup \widetilde{h}(\tau x_i)}  \subset X  $$
and hence 
$$\bigsqcup_{i \in I} H_i^{\sqcup \widetilde{h'}(\tau x_i)}  \subset X \setminus \{ \tau x_j \} \subset X' . $$
 \end{proof}

In particular,  Lemma~\ref{lem : preliminary for def of eta} implies that if $(X,h)$ is an indecomposable object in $\HQ^{(1)}$ such that $\widetilde{h}(\tau x_i) >0$ for some $i \in I$ then both $\tau x_i$ and $x$ appear (at least once) in $X$ (recall that $x$ denotes the chosen indecomposable object in $\DbRepQ$ provided by Lemma~\ref{lem : intersection of hammocks}).
 Consequently, we can define an irreducible morphism $\eta_i^{(X,h)}$ for every indecomposable object $(X,h)$ in $\HQ^{(1)}$ and every $i \in \Sigma(X,h)$ by choosing a pair of non-zero morphisms $f_i \in \mathrm{Hom}_{\DbRepQ}(\tau x_i , x)$ and $g_i \in \mathrm{Hom}_{\DbRepQ}(x , S \tau x_i)$ and then set 
 \begin{equation}
     \eta_i^{(X,h)} := 
        f_i \otimes g_i \otimes \bigotimes_{y \in X \setminus \{\tau x_i , x \}} \mathrm{id}_y 
 \end{equation}
    It remains to show that these morphisms satisfy the desired properties. 

     \begin{proof}[\textbf{Proof of Theorem~\ref{thm : existence of morphisms eta}}.]
     
Let $(X,h)$ be an indecomposable object in $\HQ^{(1)}$, $i \in \Sigma(X,h)$ and $j \in I$ such that there is an oriented path from $j$ to $i$ in $Q$. We denote $(X',h') :=  \mu_i(X,h)$. In view of Definition~\ref{def : Serre tiltings},  Lemma~\ref{lem : positivity on the section} implies 
 \begin{equation} \label{eq : positivity}
 \widetilde{h'}(\tau x_j) = \widetilde{h- \bdelta_{\tau x_i}}(\tau x_j) = \widetilde{h}(\tau x_j) + 1 >0  
  \end{equation}
which show that  $\Sigma(X',h')$ contains all $j \in I$ such that $i \leftarrow j$. Then, for any $j \in \Sigma(X',h')$ such that $j \rightsquigarrow i$, denoting $X'' := X \setminus  \{ \tau x_i, \tau x_j , x \}$ we have
\begin{align*}
    \eta_j^{(X',h')} \circ \eta_i^{(X,h)} &= \left( g_j \otimes f_j \otimes  \mathrm{id}_{S \tau x_i}  \otimes \bigotimes_{y \in X''} \mathrm{id}_y \right) \circ \left( f_i \otimes \mathrm{id}_{\tau x_j} \otimes  g_i \otimes  \bigotimes_{y \in X''} \mathrm{id}_y   \right) \\
     &= (g_j \circ f_i) \otimes f_j \otimes g_i \otimes  \bigotimes_{y \in X''} \mathrm{id}_y \\
     &= 0 
\end{align*}
      because 
      $$ g_j \circ f_i \in \mathrm{Hom}_{\DbRepQ}(\tau x_i , S \tau x_j) \simeq D\mathrm{Hom}_{\DbRepQ}(\tau x_j , \tau x_i) = 0 $$
      as there is no oriented path from $i$ to $j$ in $Q$ given that  $j \rightsquigarrow i$ and $Q$ is acyclic. This establishes the property (i) of Theorem~\ref{thm : existence of morphisms eta}. 
      
      Let us now assume $f \in \mathrm{Hom}_{\HQ^{(1)}}(\mu_i(X,h) , T)$ such that $f \circ \eta_i^{(X,h)} = 0$. Assume without loss of generality that $T$ is indecomposable. Let us write $f = a \otimes b  $  where $a \in \mathrm{Hom}_{\DbRepQ}(x,z)$ and $b$ is a tensor product of non-trivial morphisms with domains in $ \mu_i X \setminus \{ x\}$. The composition with $\eta_i^{(X,h)}$ can be trivial only via $f_i \circ a$ where $f_i \in \mathrm{Hom}_{\DbRepQ}( \tau x_i , x)$. Hence the codomain  of $a$ must be of the form $z = S \tau x_k$ with $k \rightsquigarrow i$. Moreover the set  $\{ k \in I \mid  \text{$k \rightsquigarrow j$ and $\widetilde{h'}(\tau x_k) >0$ } \} $ is not empty (as it contains for instance any $j$ such that $j \rightarrow i$, by~\eqref{eq : positivity}). For any $k$ in that set, recall that by definition of $\HQ^{(1)}$ we have that $f(\tau x_k) \in \{x\} \sqcup \{S \tau x_j , j \in I \}$. Hence we can consider the morphism $h : \mu_{ \{ \tau x_i , \tau x_k \}}(X,h) \rightarrow T$ given by 
      $$ h : S \tau x_k \rightarrow  S \tau x_k \quad x \rightarrow f(\tau x_k) \quad  \text{$y \rightarrow f(y)$ for all $ y \neq x, S \tau x_k$ }.   $$
      Then it is straightforward to check that $h\circ \eta_k^{\mu_i(X,h)} = f $. This establishes the second property of Theorem~\ref{thm : existence of morphisms eta}.

     \end{proof}


  \section{Exact chain complexes in $\KQ^{(1)}$}
 \label{sec : exact chain complexes}

   \subsection{$\HQ^{(1)}$-exact chain complexes }
    \label{sec : exactness}

 Given an \textit{indecomposable dominant} object $X$ in $\HQ^{(1)}$ that we write under the form
 $$ X := \bigotimes_{k \in I} Y(\tau x_k)^{\otimes c_k} \otimes_{\HQ^{(1)}} Y(x_k)^{\otimes d_k} $$
 where $c_k, d_k$ are non-negative integers for every $k \in I$, we define an element $\bomega(X)$ of the positive orthant of the root lattice of $\mathfrak{g}$ as follows. 
 For each $i \in I$, we set
 \begin{equation} \label{eq : def of ri}
  r_i := \sum_{ \substack{ i \rightsquigarrow j \\ \text{$j$ sink in $Q$} }} (\xi(i)-\xi(j)) .  
   \end{equation}
 Note that $r_i$ is a non-negative integer for each $i \in I$, and moreover we have that $r_i = 0$ if (and only if) $i$ is a sink in $Q$. 
 We define a family $\{a_i, i \in I\}$ of non-negative integers by induction on $r_i$. If $r_i= 0$, i.e. $i$ is a sink in $Q$, we set $a_i := \max(0,c_i-d_i)$. Then, if $i \in I$ is such that $a_j$ has been constructed for all $j \in I$ satisfying $r_j < r_i$, we set 
 $$ a_i := \max \left(  0 ,  c_i-d_i + \sum_{i \rightarrow j} a_j  \right) .  $$
  This is well-defined as one can immediately check that  $r_j < r_i$ if there is an arrow $i \rightarrow j$ in $Q$. We then define $\bomega(X)$ as 
  $$ \bomega(X) := \sum_{i \in I} a_i \alpha_i  $$
  and we denote by $\mathrm{Supp}(X)$ the subset of $I$ given by 
 $$ \mathrm{Supp}(X) := \{i \in I \mid a_i >0\} .   $$
  We can now introduce the main definitions of this section. 

   \begin{deftn} \label{def : CQ1 exact chain complexes}
  Let $C_{\bullet}$ be a chain complex in $\KQ^{(1)}$. We say that $C_{\bullet}$ is $\HQ^{(1)}$-exact if the following properties hold:
   \begin{itemize}
       \item We have that $C_n = 0$ for all $n<0$ and $C_0$ is indecomposable dominant. 
       \item For every $n>0$ and every finite set $Z \subset \ZQz$, we have that 
       $$ Z \cap \mathrm{Supp}(C_0) \neq \emptyset \quad \Rightarrow \quad \forall f : C_n \rightarrow \mu_Z(C_0), f \circ d_{n-1}^{C_{\bullet}} = 0 \Rightarrow \text{$f$ factors through $d_n^{C_{\bullet}}$.} $$
   \end{itemize}
   \end{deftn}

 \begin{rk}
     The fact that the codomain of $f$ in the second condition of Definition~\ref{def : CQ1 exact chain complexes} is written as $\mu_Z(C_0)$ is not restrictive. Indeed, by construction of the category $\HQ$, any indecomposable object $X$ such that there is a non-trivial morphism from $C_n$ to $X$ has to be a Serre tilting of an indecomposable summand of $C_n$; this summand is itself is a Serre tilting of $C_0$ in view of the sequence of morphisms $C_0 \xrightarrow[]{d_0} \rightarrow \cdots \rightarrow C_n $, and thus $X$ is a Serre tilting of $C_0$. 
 \end{rk}

  \begin{deftn} \label{def : strong exactness}
     Let $C_{\bullet}$ be a chain complex in $\KQ^{(1)}$.  We say that $C_{\bullet}$ is strongly $\HQ^{(1)}$-exact if $C_{\bullet}$ is $\HQ^{(1)}$-exact and moreover all the non-trivial components of each differential of $C_{\bullet}$ are of the form $\eta_i \otimes \mathrm{id}$ for some $i \in \mathrm{Supp}(C_0)$. 
     
   \end{deftn}

 We can show the following. 

\begin{lem} \label{lem :uniqueness up to isom}
Let  $C_{\bullet}$ and $D_{\bullet}$ be objects in $\KQ^{(1)}$. Assume that both $C_{\bullet}$ and $D_{\bullet}$ are strongly $\HQ^{(1)}$ exact  and that $C_0 \simeq D_0$. Then $C_{\bullet} \simeq D_{\bullet}$ in $\KQ^{(1)}$. 
\end{lem}

 \begin{proof}
 The assumptions obviously  imply that $\mathrm{Supp}(C_0) = \mathrm{Supp}(D_0)$. 
  Let us choose an isomorphism $u_0 : C_0 \rightarrow D_0$. The strong exactness of $C_{\bullet}$ and $D_{\bullet}$ implies that $d_0^{C\bullet}$ and $d_0^{D_{\bullet}}$ have the same non trivial components  and therefore we have that $C_1 \simeq D_1$ so we can obviously complete the square. Then the composition $ \tilde{u}_1 :=  d_1^{D_{\bullet}} \circ u_1$  satisfies $\tilde{u}_1 \circ d_0^{C_{\bullet}} = 0$. Moreover the codomain  of $\tilde{u}_1$ is $D_2$  whose indecomposable summands are all of the form $\mu_Z(D_0)$ with $Z \subset \mathrm{Supp}(D_0) = \mathrm{Supp}(C_0)$ (by strong exactness of $D_{\bullet}$. Therefore by exactness of $C_{\bullet}$, $\tilde{u}_1$ factors through $d_1^{C_{\bullet}}$ which yields a map $u_2 : C_2 \rightarrow D_2$ making the obvious square commute. Continuing this way we obtain a morphism $u_{\bullet} : C_{\bullet} \rightarrow D_{\bullet}$. The same arguments starting with an isomorphism $v_0 : D_0 \rightarrow C_0$ yields a morphism $v_{\bullet} : D_{\bullet} \rightarrow C_{\bullet}$ and we immediately have that $u_{\bullet} v_{\bullet} = \mathrm{id}_{D_{\bullet}}$ and $v_{\bullet} u_{\bullet} = \mathrm{id}_{C_{\bullet}}$ which proves the lemma.
 \end{proof}

 \subsection{Initial chain complexes}
    \label{sec : initial complexes}

 For every $i \in I$ we consider the following indecomposable objects in $\HQ^{(1)}$:
  \begin{equation} \label{eq : def of initial objects}
     H_i := Y(x_i) \qquad K_i := Y(\tau x_i) \otimes_{\HQ^{(1)}} Y(x_i) \qquad F_i := F(\tau x_i) .
  \end{equation}
  We then consider a certain collection of chain complexes in $\KQ^{(1)}$ that are non-trivial in only one homological degree. We will sometimes refer to these complexes as initial complexes. They can be grouped into three families each of cardinality $n = \sharp I$:
 \begin{itemize}
     \item For every $i \in I$, the chain complex $\mathbf{H}_i$ whose only non zero object lies in homological degree zero and is given by $Y(x_i)$,
     \item For every $i \in I$, the chain complex $\mathbf{K}_i$ whose only non zero object lies in homological degree zero and is given by $K_i$,
     \item  For every $i \in I$, the chain complex $\mathbf{F}_i$ whose only non zero object lies in homological degree one and is given by $F_i$.
 \end{itemize}

 In particular, the chain complexes belonging to the first two families are strongly $\HQ^{(1)}$-exact. We now show that the notions of exactness introduced in the previous subsection are preserved under tensor product by the chain complexes $\mathbf{K}_i$.

  \begin{lem} \label{lem : product by frozen still exact}
    Assume $C_{\bullet}$ is $\HQ^{(1)}$-exact (resp. strongly $\HQ^{(1)}$-exact). Then for any $i \in I$, $\mathbf{K}_i \otimes C_{\bullet}$ is  $\HQ^{(1)}$-exact (resp. strongly $\HQ^{(1)}$-exact).   
  \end{lem}

   \begin{proof}
    By definition, $\mathbf{K}_i$ is non-trivial only in degree $0$, so that $(\mathbf{K}_i \otimes C_{\bullet})_n = (\mathbf{K}_i)_0 \otimes C_n = K_i \otimes C_n$ for every $n \geq 0$. In particular the first condition of Definition~\ref{def : CQ1 exact chain complexes} is satisfied. Moreover, the differentials of $\mathbf{K}_i \otimes C_{\bullet}$ are given by $d_n^{\mathbf{K}_i \otimes C_{\bullet}   }  = \mathrm{id}_{K_i} \otimes d_n^{C_{\bullet}}$. 

     It follows from the definition of $\bomega$ that $\bomega(K_i \otimes C_0) = \bomega(C_0)$ so that $\mathrm{Supp}(K_i \otimes C_0) = \mathrm{Supp}(C_0)$. Let us fix $n>0$ and let $Z$ be a finite subset of $\ZQz$ such that $Z \cap \mathrm{Supp}(C_0)  \neq \emptyset$. Let $f : K_i \otimes C_n \rightarrow \mu_Z(K_i \otimes C_0)$ be a non-trivial morphism and assume  that $f \circ (\mathrm{id}_{K_i} \otimes d_{n-1}^{C_{\bullet}} = 0$. This implies that (up to composing with an isomorphism) $f$ can be written under the form $f = g \otimes h$ where $g : K_i \rightarrow \mu_{Z'}K_i$ and $h : C_n \rightarrow \mu_{Z''}C_0$ with $Z' \sqcup Z'' = Z$ and such that $h \circ d_{n-1}^{C_{\bullet}} = 0$. As $C_{\bullet}$ is $:CQ^{(1)}$-exact, we get that $h = h' \circ d_n^{C_{\bullet}}$ and hence we obtain 
     $$ f =  g \otimes (h' \circ d_n^{C_{\bullet}}) = (g \otimes h') \circ (\mathrm{id}_{K_i} \otimes d_n^{C_{\bullet}}) = (g \otimes h') \circ d_n^{\mathbf{K}_i \otimes C_{\bullet}} .   $$
     This shows that $\mathbf{K}_i \otimes C_{\bullet}$ is $\HQ^{(1)}$-exact. The strong exactness is obvious. 
   \end{proof}

\subsection{Leading  objects}
 \label{sec : (co)leading objects}

 Given an element $\beta := \sum_{i \in I} a_i \alpha_i$ in the positive orthant of the root lattice of $\mathfrak{g}$, we define integers $b_i, c_i, d_i$ for each $i \in I$ as follows: 
    $$ \forall i \in I, \enspace  b_i := a_i - \sum_{i \rightarrow j} a_j \quad c_i := (b_i)_+ \quad d_i := (b_i)_-  $$
    Here we are using the standard notations $b_+ := \max(0,b) , b_- := - \min(0,b)$ for any integer $b$. In particular we have that $c_i \geq 0, d_i \geq 0, c_id_i = 0$ and $c_i - d_i = b_i$  for every $i \in I$.
    We then define the following indecomposable dominant object in $\HQ^{(1)}$:
    \begin{equation} \label{eq : def of Ybeta leading}
        \Ybeta := \bigotimes_{i \in I} Y(\tau x_i)^{\otimes c_i} \otimes Y(x_i)^{\otimes d_i} . 
    \end{equation}
   
 We can already make the following observation. 

  \begin{lem} \label{lem : support}
    For every $\beta := \sum_{i \in I} a_i \alpha_i$ in the positive orthant of the root lattice of $\mathfrak{g}$, we have that  $\bomega(Y[\beta]) = \beta$. 
  \end{lem}

   \begin{proof}
       Denote $\bomega(Y[\beta]) := \sum_i a'_i \alpha_i$. We show that $a'_i = a_i$ for all $i \in I$ by induction on $r_i$ as defined by~\eqref{eq : def of ri}. If $i \in I$ is such that $r_i=0$, i.e. $i$ is a sink in $Q$, then by definition of $\bomega$ we have that $a'_i = c_i-d_i = (b_i)_+ - (b_i)_- = b_i = a_i$ (as the sum over $j$ such that $i \rightarrow j$ is empty given that $i$ is a sink in $Q$). Let $i \in I$ and assume that $a'_j=a_j$ has been proved for all $j \in I$ such that $r_j<r_i$. Then by definition of $\bomega$ we have 
    \begin{align*}
        a'_i &:= \max \left( 0, c_i-d_i + \sum_{i \rightarrow j} a'_j \right) = \max \left( 0 , b_i + \sum_{i \rightarrow j} a'_j \right)  \\
        &= \max \left( 0 , b_i + \sum_{i \rightarrow j} a_j \right) \quad \text{by induction assumption, as $r_j<r_i$ if $i \rightarrow j$} \\
        &= \max(0,a_i)  \quad \text{by definition of $b_i$} \\
        &= a_i
    \end{align*}
    as $\beta$ belongs to the positive orthant of the root lattice. The lemma is proved. 
   \end{proof}

    From this lemma, we can derive the following useful observation. 

     \begin{cor} \label{cor : dominant object as leading object}
     Let $(X,h)$ be an indecomposable  dominant object in $\HQ^{(1)}$ and let $\beta := \bomega(X,h)$. Then we have 
     $$ (X,h) \enspace \simeq \enspace \bigotimes_{i \in I} K_i^{\otimes \min( \widetilde{h}(\tau x_i) , \widetilde{h}(x_i) )} \otimes  Y[\beta] .  $$
     \end{cor}

\begin{proof}
    As $(X,h)$ is indecomposable dominant we can write it as 
    $$ (X,h) \simeq \bigotimes_{i \in I} Y(\tau x_i)^{\otimes c_i} \otimes Y(x_i)^{\otimes d_i} $$
    with $c_i,d_i \geq 0$ for each $i \in I$. Moreover, we actually have $c_i = \widetilde{h}(\tau x_i)$ and $d_i = \widetilde{h}(x_i)$ for each $i \in I$. So we can write 
    $$ (X,h) \simeq  \bigotimes_{i \in I} K_i^{\otimes \min( \widetilde{h}(\tau x_i) , \widetilde{h}(x_i) )} \otimes  (X',h')  $$
    where denoting $c'_i = \widetilde{h'}(\tau x_i)$ and $d'_i := \widetilde{h'}(x_i)$ we have $c'_i-d'_i = c_i-d_i$ and  $c'_id'_i=0$, for every $i \in I$. As the definition of $\bomega$ depends only on the integers $c_k-d_k$, we get $ \bomega (X',h') = \beta = \bomega(\Ybeta)$ by Lemma~\ref{lem : support} which implies that $(X',h') \simeq \Ybeta$ because the conditions $c'_id'_i=0$ determine the integers $c_i,d_i$ uniquely for a given $\bomega$. 
 \end{proof}

 \begin{rk} \label{rk : support of Ybeta}
     Recalling the notations $\mathrm{Supp}$ from Section~\ref{sec : exactness}, it follows from Lemma~\ref{lem : support} that $\mathrm{Supp}(Y[\beta]) = \{i \in I , a_i \neq 0\}$ if $\beta := \sum_i a_i \alpha_i $ (which was the motivation for the definition of $\bomega(X)$ and $\mathrm{Supp}(X)$ in Section~\ref{sec : exactness}).
 \end{rk}

   \subsection{Serre tiltings of leading objects}
    \label{sec : morphisms between objects}

Let $Q$ be a finite quiver. We assume that for any vertices $i,j$ of $Q$, there is at most one oriented path from $i$ to $j$ (this is the case for instance if $Q$ is a Dynkin quiver). Given an element $\beta$ of the positive orthant of the root lattice, we denote by $\mathrm{Supp}(\beta) := \{ i \in I \mid a_i>0 \}$ and by $\Qbeta$ the subquiver of $Q$ having $\mathrm{Supp}(\beta)$ as set of vertices. Let $P_i^{\beta}$ (resp. $I_i^{\beta}$) denote the   the  indecomposable projective (resp. injective) representation of $\Qbeta$ at the vertex $i$. Recall that their respective dimension vectors are given by 
$$ \boldsymbol{\dim} P_i^{\beta} = \sum_{ j \in \mathrm{out}_{\Qbeta}(i)} \alpha_j  \qquad 
  \boldsymbol{\dim} I_i^{\beta} = \sum_{ j \in \mathrm{in}_{\Qbeta}(i)} \alpha_j . $$
 where for each $i \in (\Qbeta)_0$:
$$ \mathrm{out}_{\Qbeta}(i) := \{ j \in (\Qbeta)_0 \mid i \rightsquigarrow j \}  \qquad \mathfrak{in}_{\Qbeta}(i) := \{ j \in (\Qbeta)_0 \mid j \rightsquigarrow i \} .  $$
The ultimate goal of this section will be to construct morphisms of chain complexes between various $\HQ^{(1)}$-exact complexes in $\KQ^{(1)}$. The first step to achieve this consists in constructing morphisms between (Serre tiltings of) leading objects in $\HQ^{(1)}$, which is the aim of the next few statements. 

 \begin{lem} \label{lem : isom between leading objects}
     We have that 
     $$ \Ybeta \otimes Y(x_i) \simeq  K_i^{\otimes \epsilon_i} \otimes  Y[\beta - \boldsymbol{\dim}I_i^{\beta}] $$
     where $\epsilon_i$ is equal to  $1$ if $c_i>0$ and to zero otherwise. 
 \end{lem}

  \begin{proof}
    Let us write $\beta = \sum_k a_k \alpha_k$ and
    $$ \Ybeta = \bigotimes_{k \in I} Y(\tau x_k)^{\otimes c_k} \otimes Y(x_k)^{\otimes d_k} .  $$
     We claim  that $\bomega(\Ybeta \otimes Y(x_i))= \beta - \boldsymbol{\dim}I_i^{\beta}$. To check this, let us write  
     $$\bomega(\Ybeta \otimes Y(x_i))  :=  \sum_{k \in I}a'_k \alpha_k  \qquad \Ybeta \otimes Y(x_i) := \bigotimes_{k \in I} Y(\tau x_k)^{\otimes c'_k} \otimes Y(x_k)^{\otimes d'_k} $$
     so we obviously have $c'_k = c_k$ and $d'_k = d_k + \delta_{k,i}$ for each $k \in I$, and let us prove by induction on $r_k$ that $a'_k = a_k - \delta_{k \in \mathrm{in}_{\Qbeta}(i)}$. Assume $r_k=0$ i.e. $k$ is a sink in $Q$. Then 
     $$ a'_k = c'_k-d'_k  = c_k - (d_k + \delta_{k,i}) = a_k - \delta_{k,i} .  $$
     This implies the desired identity in that case. Assume it has been established for all $k \in I$ such that $r_k \leq m$ and let us fix $k \in I$ with $r_k = m+1$. Then $r_j \leq m$ for each $j \in I$ such that $k \rightarrow j$ and therefore we have that 
    \begin{align*}
      a'_k &= \max \left( 0 , c'_k - d'_k + \sum_{ k \rightarrow j} a'_j \right)  \\
      &= \max \left( 0 , c_k - (d_k + \delta_{k,i}) + \sum_{k \rightarrow j} (a_j - \delta_{j \in \mathrm{in}_{\Qbeta}(i)}) \right) \\
      &= \max \left( 0 , c_k-d_k + \sum_{k \rightarrow j}a_j  - \delta_{k \in \mathrm{in}_{\Qbeta}(i)} \right) . 
    \end{align*}
   Indeed, by assumption on $Q$ there is at most one oriented path from $k$ to $i$ and therefore if $k \neq i$ then $k \in \mathrm{in}_{\Qbeta}(i)$ if and only if there is $j \in \mathrm{in}_{\Qbeta}(i)$ such that $k \rightarrow j$ and if so then such $j$ is unique. It remains to note that if  $k \in \mathrm{in}_{\Qbeta}(i)$ then in particular $k \in \mathrm{Supp}(\beta)$ so that $ c_k - d_k + \sum_{k \rightarrow j} a_j = a_k >0$. Hence we have that $a'_k =  a_k - \delta_{k \in \mathrm{in}_{\Qbeta}(i)}$ which was the desired identity. Thus the claim holds.  Finally, we have that $c'_kd'_k = c_kd_k=0$ for all $k \in I$ except for $k=i$ where this the case only if $c_i=0$. If on the contrary $c_i>0$ then necessarily $d_i=0$ so that $d'_i=1$. By Corollary~\ref{cor : dominant object as leading object} we obtain 
   $$ \Ybeta \otimes Y(x_i) \otimes K_i^{\otimes \epsilon_i} \otimes Y[\beta - \boldsymbol{ \dim }I_i^{\beta}] $$
    as desired. 
  \end{proof}

\begin{prop} \label{prop : Serre tilting  of leading object}
   Let $\beta$ be an arbitrary element of the positive orthant of the root lattice and let $i \in \mathrm{Supp}(\beta)$. Let $\mathrm{out}_{\Qbeta}(i) := \{ j \in \mathrm{Supp}(\beta)  \mid i \rightsquigarrow j\}$  and let $Z$ be the finite subset of $\ZQz$ given by $Z := \{ \tau x_j  , j \in \mathrm{out}_{\Qbeta}(i)\}$. Then in $\HQ^{(1)}$, we have an isomorphism 
   $$ \mu_Z \left(  \Ybeta \otimes Y(x_i)  \right) \enspace \simeq \enspace F(i,\beta) \otimes K(i, \beta) \otimes H(i,\beta) \otimes   Y[\beta -  \boldsymbol{\dim}P_i^{\beta}] .  $$
   where 
   $$ F(i,\beta) :=  \bigotimes_{z \in Z} F(z)  \qquad H(i,\beta) :=  \bigotimes_{l \notin \mathrm{Supp}(\beta)}  Y(x_l)^{\otimes  \sharp  j \in \mathrm{out}_{\Qbeta}(i) , j \rightarrow l } $$
   and 
   $K(i, \beta)$ is a product of objects of the form $K_j , j \in I$.
 \end{prop}

\begin{proof}
 We let $c_k,d_k$ denote the non-negative integers such that $\Ybeta \simeq \bigotimes_{k \in I}Y(\tau x_k)^{\otimes c_k} \otimes Y(x_k)^{\otimes d_k}$. 
Consider the following indecomposable dominant object in $\HQ^{(1)}$:
$$ (X,h) := \Ybeta \otimes \bigotimes_{z \in Z}Y(\tau^{-1}z) .$$
In particular we have that $\widetilde{h}(\tau x_k) = c_k$ and $\widetilde{h}(x_k) = d_k + \delta_{k \in \mathrm{out}_{\Qbeta}(i)}$ for each $k \in I$. First of all, we claim that 
 \begin{equation} \label{eq : isom with Serre tilting}
\mu_Z(X,h) \simeq \bigotimes_{z \in Z}F(z) \otimes (X',h') 
\end{equation}
with $(X',h')$  dominant in $\HQ^{(1)}$. To see this, we perform the Serre tiltings on the elements of $Z$ by starting at the vertices $j \in \mathrm{out}_{\Qbeta}(i)$ such that $j$ is a sink in $\Qbeta$ and then following the orientation backwards (ending with the Serre tilting at $i$). Looking at the definition of $\Ybeta$, we see that if $j$ is a sink in $\Qbeta$ then $c_j>0$ (and $d_j=0$) so we can write $(X,h)$ as
$$ (X,h) \simeq \bigotimes_{\substack{j \in \mathrm{out}_{\Qbeta}(i) \\ \text{$j$ sink in $\Qbeta$}}} K_j \otimes \bigotimes_{ \substack{l \in \mathrm{out}_{\Qbeta}(i) \\  \text{$l$ not sink in $\Qbeta$}}} Y(x_l) \otimes (X'',h'')  $$
where $(X'',h'') := = \bigotimes_{k \in I} Y(\tau x_k)^{\tilde{c}_k} Y(x_k)^{d_k}$  with $\tilde{c}_k = c_k - 1$ if $k$ is a sink in $\Qbeta$ with $i \rightsquigarrow k$, $c'_k = c_k$ otherwise. Note in particular that  $(X'',h'')$ is dominant. Then, Lemma~\ref{lem : mutation for objects in CQ} implies that every Serre tilting  $\mu_jK_j$ for $j$ sink  in $\Qbeta$ with $i \rightsquigarrow j$ will be isomorphic to $F_j \otimes Y(\tau x_l)$  tensored by a dominant object, where $l \in \mathrm{out}_{\Qbeta}(i)$ such that $l \rightarrow j$.  Therefore after these Serre tiltings we will obtain 
$$ \bigotimes_{\substack{j \in \mathrm{out}_{\Qbeta}(i) \\ \text{$j$ sink in $\Qbeta$}}} \left(  F_j \otimes \bigotimes_{ \substack{l \in \mathrm{out}_{\Qbeta}(i) \\ l \rightarrow j}} K_l  \right) \otimes \cdots  $$
where $\cdots$ is dominant. Continuing this process, we see that at each step, we always perform Serre tiltings of the form $\mu_j K_j$ which is isomorphic to $F_j$ tensored by a dominant object, due to Lemma~\ref{lem : mutation for objects in CQ}. This proves the claim.

 As $(X',h')$ is dominant, it is entirely determined (up to isomorphism) by $h'$ (cf. Remark~\ref{rk on dominant objects}). Let us denote by $c'_k := \widetilde{h'}(\tau x_k)$ and $d'_k := \widetilde{h'}(x_k)$ for every $k \in I$. It follows from the definition of Serre tiltings that 
 $$ c'_k = c_k - \delta_{k \in \mathrm{out}_{\Qbeta}(i)} + \sharp \{  j \in \mathrm{out}_{\Qbeta}(i),  k \rightarrow j \} $$
 $$ d'_k = d_k  + \sharp \{  j \in \mathrm{out}_{\Qbeta}(i),  j \rightarrow k \}  . $$
 Note that if $k \in \mathrm{Supp}(\beta)$ then there can be an arrow $j \rightarrow k$ with $j \in \mathrm{out}_{\Qbeta}(i)$ only if $k \in \mathrm{out}_{\Qbeta}(i) \setminus \{i \}$. In other words, we have $d'_k>d_k$ if only if either  $k \in \mathrm{out}_{\Qbeta}(i) \setminus \{i \}$ or $k \notin \mathrm{Supp}(\beta)$. Therefore we can rewrite $(X',h')$ as 
  $$  (X'',h'') \otimes H(i, \beta) \otimes  \bigotimes_{z \in Z \setminus \{ \tau x_i\}} Y(\tau^{-1} z)$$
   with $(X'',h'')$ is the dominant object given by 
  $$  (X'',h'') := \bigotimes_{k \in I} Y(\tau x_k)^{\otimes c''_k}  \otimes Y(x_k)^{\otimes d''_k} $$
  with 
  $$   c''_k := c'_k \quad \text{and} \quad  d''_k := \begin{cases}
      d_k -1 + \sharp \{ j \in \mathrm{out}_{\Qbeta}(i) , j \rightarrow k \} & \text{if $k \in \mathrm{out}_{\Qbeta}(i) \setminus \{i\}$,} \\
      d_k & \text{otherwise.}
  \end{cases}    $$
  Hence~\eqref{eq : isom with Serre tilting} becomes 
  $$ \mu_Z(\Ybeta \otimes Y(x_i)) \simeq   F(i, \beta) \otimes  H(i, \beta) \otimes (X'',h'') . $$
 We are now going to show that $\bomega((X'',h'')) = \beta - \boldsymbol{\dim}P_i^{\beta}$. Namely, denoting $\bomega((X'',h'')) := \sum_{k \in I} a''_k \alpha_k$, we prove by induction on $r_k$ that $a''_k = a_k - \sharp \{\text{paths from $i$ to $k$ in $\Qbeta$} \}$ for each $k \in I$. Assume $k \in I$ is such that $r_k=0$ i.e. $k$ is a sink in $Q$. Then obviously $\sharp \{ j \in \mathrm{out}_{\Qbeta}(i) , k \rightarrow j \} = 0$. Moreover, by assumption on $Q$, there is at most one oriented path from $i$ to $k$ so that $ \sharp \{ j \in \mathrm{out}_{\Qbeta}(i) , j \rightarrow k \}  $ is equal to $1$ if $k \in \mathrm{out}_{\Qbeta}(i)$ and is equal to zero otherwise. So, if $k \notin \mathrm{out}_{\Qbeta}(i)$ then $a''_k = c''_k-d''_k = c_k-d_k = a_k$. If $k=i$ then $a''_k = c''_k-d''_k = c_k-1-d_k = a_k-1$. Finally  if $k \in \mathrm{out}_{\Qbeta}(i) \setminus \{i\}$ then 
 \begin{align*}
    a''_k &= c''_k - d''_k = (c_k - 1) - \left(  d_k-1 + \sharp \{ j \in \mathrm{out}_{\Qbeta}(i) , j \rightarrow k \} \right) \\
    &= a_k - 1 .  
 \end{align*}
  Thus the desired identity holds in that case. Assume now that it has been proved for all $k \in I$ such that $r_k \leq m$ and let us fix $k$ such that $r_k = m+1$. 
   Assume $k \notin \mathrm{out}_{\Qbeta}(i)$. Then we have $r_j<r_k$ for every $j \in I$ such that $k \rightarrow j$ and thus using the induction assumption we get
   \begin{align*}
       a''_k &= \max \left( 0 , c''_k-d''_k + \sum_{k \rightarrow j} a''_j \right) \\
       &= \max \left( 0 , c_k -d_k + \sharp \{ j \in \mathrm{out}_{\Qbeta}(i) , k \rightarrow j \}  + \sum_{k \rightarrow j} (a_j - \delta_{ j \in \mathrm{out}_{\Qbeta}(i)} \right) \\
       &= \max \left( 0 , c_k-d_k + \sum_{k \rightarrow j}a_j \right) = a_k .  
   \end{align*}
 Assume now $k=i$. Then similarly we get 
 \begin{align*}
      a''_i &= \max \left( 0  , c''_i-d''_i + \sum_{i \rightarrow j} a''_j \right) \\
      &=  \max \left( 0, c_i-1-d_i + \sharp \{ j \in (\Qbeta)_0   , i \rightarrow j \} +  \sum_{i \rightarrow j} (a_j - \delta_{j \in (\Qbeta)_0})  \right)  \\
      &= \max \left( 0 , c_i-d_i-1 + \sum_{i \rightarrow j} a_j \right) . 
 \end{align*}
 As $i \in \mathrm{Supp}(\beta)$, we have that $a_i>0$ so that $a_i = c_i-d_i + \sum_{i \rightarrow j} a_j$ and thus we obtain 
 $$ a''_i = \max(0,a_i-1) = a_i-1 .  $$
  Finally  if $k \in \mathrm{out}_{\Qbeta}(i) \setminus \{i \}$ then 
 \begin{align*}
    c''_k-d''_k + \sum_{k \rightarrow j} a''_j &=   c_k -1 + \sharp \{ j \in \mathrm{out}_{\Qbeta}(i) , k \rightarrow j\} - d_k +1 \\
    & \qquad \qquad - \sharp \{ j \in \mathrm{out}_{\Qbeta}(i) , j \rightarrow k\} + \sum_{k \rightarrow j}(a_j - \delta_{j \in \mathrm{out}_{\Qbeta}(i)} ) \\
    &= c_k-d_k + \sum_{k \rightarrow j} a_j - \sharp \{ j \in \mathrm{out}_{\Qbeta}(i) , j \rightarrow k\} . 
 \end{align*}
 By assumption on $Q$, there can only be at most one $j \in \mathrm{out}_{\Qbeta}(i)$ such that $ j \rightarrow k$, and in this case there is necessarily one given that $k \in \mathrm{out}_{\Qbeta}(i) \setminus \{i\}$. Hence we obtain 
 $$ a''_k = \max \left( 0 , a_k-1 \right)  = a_k-1 $$
 as $a_k>0$ given that $k \in \mathrm{Supp}(\beta)$. 
  Hence we have proved that $\bomega((X'',h'')) = \beta - \boldsymbol{ \dim}P_i^{\beta}$. The conclusion follows using Corollary~\ref{cor : dominant object as leading object}. 
\end{proof}

 \begin{cor} \label{cor : morphism between leading objects}
Let $i \in \mathrm{Supp}(\beta)$ and let $Z$ denote the finite subset of $\ZQz$ given by $Z := \{ \tau x_j , j \in \mathrm{out}_{\Qbeta}(i) \}$. Then we have that    
$$  \mathrm{Hom}_{\HQ^{(1)}} \left(  \mu_{Z \setminus \{ \tau x_i \}}( \Ybeta \otimes Y(x_i))  \enspace ,  \enspace F(i,\beta) \otimes K(i, \beta) \otimes H(i,\beta) \otimes   Y[\beta -  \boldsymbol{\dim}P_i^{\beta}]  \right) \neq 0.   $$
 \end{cor}

\begin{proof}
Denoting $(X,h) := \mu_{Z \setminus \{ \tau x_i \}}( \Ybeta \otimes Y(x_i))$, Proposition~\ref{prop : Serre tilting  of leading object} tells us that $\mu_i(X,h) \simeq F(i,\beta) \otimes K(i, \beta) \otimes H(i,\beta) \otimes   Y[\beta -  \boldsymbol{\dim}P_i^{\beta}] $. Moreover we have (by definition of Serre tiltings)
$$ h = h_{\beta} - \sum_{j \in \mathrm{out}_{\Qbeta}(i)} \delta_{\tau x_j} $$
where $h_{\beta}$ denotes the quasi-additive function of $\Ybeta$. Hence we have 
\begin{align*}
    \widetilde{h}(\tau x_i) &= \widetilde{h_{\beta}}(\tau x_i) + \sharp \{ j \in \mathrm{out}_{\Qbeta}(i) , \rightarrow j \} = \begin{cases}
        \widetilde{h_{\beta}}(\tau x_i) & \text{if $i$ is a sink in $\Qbeta$,} \\
        \widetilde{h_{\beta}}(\tau x_i) + 1 & \text{otherwise. }
    \end{cases}
\end{align*}
 It follows from the definition of $\Ybeta$ that if $i$ is a sink in $\Qbeta$ then $\widetilde{h_{\beta}}(\tau x_i)>0$ and hence $\widetilde{h}(\tau x_i)>0$. If on the other hand $i$ is not a sink in $\Qbeta$ then Lemma~\ref{lem : positivity on the section} implies that $\widetilde{h}(\tau x_i) \geq 1$ so finally we have $\widetilde{h}(\tau x_i)>0$ in both cases. Hence the desired statement, by definition of the morphisms in the category $\HQ$.  
\end{proof}

   \subsection{Combinatorial setup}
    \label{sec : R and i(beta)}

   We assume as in Section~\ref{sec : morphisms between objects} that $Q$ is such a finite quiver such that there is at most one oriented path from one vertex to another. Given an element $\beta = \sum_i a_i \alpha_i$ of the positive orthant of the  corresponding root lattice, we denote by $\mathrm{Supp}(\beta) := \{ i \in I \mid a_i >0 \}$ and by $\Qbeta$ the subquiver of $Q$ having $\mathrm{Supp}(\beta)$ as set of vertices. Adapting~\eqref{eq : def of ri}, we define 
 $$ \forall i \in \mathrm{Supp}(\beta), \enspace r_i^{\beta} := \sum_{ \substack{j \in \mathrm{Supp}(\beta)\\ i \rightsquigarrow j}} (\xi(i)-\xi(j)) . $$ 
 We also set 
 $$ M (\beta) := \{ i \in I \mid a_i = \min_{j \in \mathrm{Supp}(\beta)} a_j \} . $$
 We then define  a subset $I(\beta)$ of $M(\beta)$ as 
 $$ I(\beta) := \{ i \in M(\beta) \mid  r_i^{\beta} = \min_{j \in M(\beta)} r_j^{\beta} \} .  $$

\subsection{Construction of the chain complexes $\Cbeta$}
 \label{sec : construction of Cbeta}

 Recall the chain complexes $\mathbf{K}_i, \mathbf{H}_i$ and $\mathbf{F}_i$ defined in Section~\ref{sec : initial complexes}. Moreover, given $\beta$ in the positive orthant of the root lattice and $i \in \mathrm{Supp}(\beta)$ we denote by $\mathbf{F}(i,\beta) , \mathbf{H}(i, \beta)$ and $\mathbf{K}(i, \beta)$ the chain complexes defined in terms of the initial chain complexes in the same way that $F(i, \beta) , H(i, \beta) $ and $K(i, \beta)$ are defined in terms of the $F_i,H_i,K_i$. In particular $\mathbf{F}(i, \beta)$ is concentrated in degree $\sharp \mathrm{out}_{\Qbeta}(i)$ while $\mathbf{H}(i, \beta)$ and $\mathbf{K}(i, \beta)$ are concentrated in degree $0$.  
 We now construct a strongly $\HQ^{(1)}$-exact chain complex $\Cbeta \in \KQ^{(1)}$ for each $\beta \leq \theta$ in the positive orthant of the root lattice. More precisely, we establish the following. 
 
 \begin{thm} \label{thm : chain complexes Cbeta}
   For every element $\beta$ of the positive orthant of the root lattice, there is a unique strongly $\HQ^{(1)}$ exact chain complex $\Cbeta$ in $\KQ^{(1)}$ such that $C_0[\beta] \simeq \Ybeta$ and such that, if $i \in \mathrm{Supp}(\beta)$ and $Z := \{ \tau x_j , j \in \mathrm{out}_{\Qbeta}(i)\}$ then  the indecomposable object $\mu_Z(\Ybeta)$ appears as summand of $C_m[\beta]$, where $m := \sharp Z$. 
 \end{thm}

 The proof of this result goes by induction on the height of $\beta$. Let us fix $\beta$ of height $r \geq 1$ and assume $\Cbeta $ has been constructed for every element of the positive orthant of height strictly less than $r$. We then show the following:

  \begin{prop} \label{prop : morphism of chain complexes}
     Let $i \in I(\beta)$ with the notations of Section~\ref{sec : R and i(beta)}. Then there is a non-trivial morphism in $\KQ^{(1)}$:
     $$ u_{\bullet}^{\beta} : \enspace \mathbf{K}_i^{\otimes \epsilon_i} \otimes C_{\bullet}[\beta - \boldsymbol{\dim}I_i^{\beta}][-1] \longrightarrow \mathbf{F}(i, \beta) \otimes  \mathbf{H}(i, \beta) \otimes  \mathbf{K}(i, \beta) \otimes  C_{\bullet}[\beta - \boldsymbol{\dim}P_i^{\beta}] . $$
  \end{prop}

 \begin{proof}
 Denote respectively  by $C_{\bullet}$ and $D_{\bullet}$ the domain and codomain of the claimed morphism. Assume  first that $i$ is a sink in $\Qbeta$. Then both $C_{\bullet}$ and $D_{\bullet}$ are trivial in degrees less or equal than $0$ and their respective degree $1$ objects are indecomposable in $\HQ^{(1)}$. Combining Lemma~\ref{lem : isom between leading objects} together with Corollary~\ref{cor : morphism between leading objects}, we obtain a non-trivial morphism $C_1 \rightarrow D_1$. Moreover this morphism is of the form $\eta_i \otimes \mathrm{id}$. Consider the composition $d_1^{D_{\bullet}} \circ u_1$. If $g$ is a  non-trivial component of $d_1^{D_{\bullet}}$, then by strong exactness of $D_{\bullet}$ we have that $g$ is of the form $\eta_j \otimes \mathrm{id}$ for some $j \in \mathrm{Supp}(D_1) =  \mathrm{Supp}(\beta- \alpha_i)$. In particular we have that $j \in \mathrm{Supp}(\beta)$.  If there is no oriented path from $j$ to $i$ in $Q$ then $j \in \mathrm{Supp}(\beta-\sum_{j \rightsquigarrow i} \alpha_j) = \mathrm{Supp}(C_1)$ and hence $g \circ u_1$ factors through $d_1^{C_{\bullet}}$ by $\HQ^{(1)}$-exactness of $C_{\bullet}$. If on the other hand $j \rightsquigarrow i$ then the composition will be trivial by Theorem~\ref{thm : existence of morphisms eta}(i). 

  If now $i$ is not a sink in $\Qbeta$ we first need to check that the morphism $u = \eta_i \otimes \mathrm{id}$ given by Corollary~\ref{cor : morphism between leading objects} can be chosen in such a way that $ u \circ d_{m-1}^{C_{\bullet}} = 0$. The induction assumption on $C_{\bullet}[\beta - \boldsymbol{\dim}I_i^{\beta}]$ implies that $\mu_{Z \setminus \{ \tau x_i\}}(\Ybeta \otimes Y(x_i)) $ appears as indecomposable summand in $C_{m-1}$. In particular the components of $d_{m-1}^{C_{\bullet}}$ having $\mu_{Z \setminus \{ \tau x_i\}}(\Ybeta \otimes Y(x_i)) $ as codomain are either trivial or of the form $\eta_j \otimes \mathrm{id}$ with $j \in \mathrm{out}_{\Qbeta}(i) \setminus \{i\}$, by strong exactness of $C_{\bullet}$. Hence it suffices to choose the components of $u$ to be $\eta_i \otimes \mathrm{id}$ on $\mu_{Z \setminus \{ \tau x_i\}}(\Ybeta \otimes Y(x_i)) $ and trivial on the other summands of $C_{m-1}$. Theorem~\ref{thm : existence of morphisms eta}(i) will then imply that the desired composition is trivial. The fact that the morphism $C_{m-1} \rightarrow D_{m-1}$ lifts to a morphism of chain complexes follows from the exactness of $C_{\bullet}$ and $D_{\bullet}$ in the same way as in the case where $i$ is a sink in $\Qbeta$. 
 \end{proof}

Recall the complex $\mathbf{H}_i$ defined in Section~\ref{sec : initial complexes}. The following is immediate. 

 \begin{lem} \label{lem  : renorm by H_i}
   The mapping cone of the morphism $u_{\bullet}^{\beta}$ above can be written as $C_{\bullet} \otimes \mathbf{H}_i$ for some chain complex  $C_{\bullet}$ in $\KQ^{(1)}$.   
 \end{lem}

 In view of Lemma~\ref{lem  : renorm by H_i}, we can define $\Cbeta$ as the (unique up to isomorphism) chain complex in $\KQ^{(1)}$ such that 
 $$ \Cbeta \otimes \mathbf{H}_i \simeq \mathrm{Cone}(u_{\bullet}^{\beta} ) . $$
 We now check that $\Cbeta$ satisfies the expected homological properties. 

  \begin{prop} \label{prop : Cbeta as mapping cone}
     The complex $\Cbeta$ is strongly $\HQ^{(1)}$-exact in the sense of Definition~\ref{def : strong exactness}. Moreover, its degree zero object is isomorphic to $\Ybeta$.   
  \end{prop}

  \begin{proof}
 Denote respectively by $C_{\bullet}$  and $D_{\bullet}$ the domain and codomain of $u_{\bullet}^{\beta}$ and by $E_{\bullet}$ its mapping cone. In particular we have that $E_n \simeq 0$ for all $n<0$ and also
 $$ E_0 \simeq  C_1 \simeq \Ybeta \otimes Y(x_i) $$
by Lemma~\ref{lem : isom between leading objects} so that $C_0[\beta]$ is indecomposable dominant and isomorphic to $\Ybeta$. Note also that by Lemma~\ref{lem : product by frozen still exact},  both $C_{\bullet}$ and $D_{\bullet}$ are (up to homological shift) $\HQ^{(1)}$-exact in the sense of Definition~\ref{def : CQ1 exact chain complexes}.
 Now let $n>0$, let $Z$ be a finite subset of $\ZQz$ such that $Z \cap \mathrm{Supp}(\beta) \neq \emptyset$ and assume we have a non-trivial morphism $f : C_n[\beta] \rightarrow \mu_Z(C_0[\beta])$ such that $f \circ d_{n-1}^{\Cbeta}= 0$. We claim that $f$ factors through $d_n^{\Cbeta}$. Indeed, $g := f \otimes \mathrm{id}_{Y(x_i)}$ is a morphism $E_n  \rightarrow \mu_Z(E_0)$ satisfying $g \circ d_{n-1}^{E_{\bullet}} = 0$. Recall that by definition of $E_{\bullet}$ we have 
 $$ E_n \simeq C_{n+1} \oplus D_n \quad \text{and} \quad d_{n}^{E_{\bullet}} = \begin{pmatrix}
     d_{n+1}^{C_{\bullet}} & 0 \\
     u_{n+1} & - d_n^{D_{\bullet}} . 
 \end{pmatrix} $$
 Thus writing $g = (h \enspace k)$ with $h : C_{n+1} \rightarrow \mu_Z(E_0)$ and $k : D_n \rightarrow \mu_Z(E_0)$ we get 
  \begin{equation} \label{eq : vanishing identities}
 h d_n^{C_{\bullet}} + k u_n^{\beta} = 0 \qquad k d_{n-1}^{D_{\bullet}} = 0 . 
   \end{equation}
Any indecomposable summand of $C_n$  is a Serre tilting of $C_1$, and this also true for any indecomposable summand of $D_n$ because of Lemma~\ref{lem : isom between leading objects} and Proposition~\ref{prop : Serre tilting  of leading object}. So, if $ \tau x_i \notin Z$ then necessarily $k=0$. As $f$ was assumed non trivial, we must have $h \neq 0$ so necessarily  $Z \cap \mathrm{Supp}(C_1) \neq 0$. Therefore $h \circ d_n^{C_{\bullet}} = 0$ implies that $h$ factors through $d_{n+1}^{C_{\bullet}}$ and so the claim holds in that case. If on the contrary $\tau x_i \in Z$, then we may write $Z = \{\tau x_i \} \sqcup Z'$.  As $Z \cap \mathrm{Supp}(\beta- \alpha_i) = Z \cap D_1 \neq \emptyset$, the second identity from~\eqref{eq : vanishing identities} implies that $k$ factors through $d_n^{D_{\bullet}}$ (by $\HQ^{(1)}$ exactness of $D_{\bullet}$ so we write $k = k' \circ d_n^{D_{\bullet}}$. Plugging this into the first identity of~\eqref{eq : vanishing identities} we obtain 
 $$ h \circ d_n^{C_{\bullet}} = - k' \circ d_n^{D_{\bullet}} \circ u_n^{\beta} = - k' \circ u_{n+1}^{\beta} \circ d_n^{C_{\bullet}}  $$
 so that $(h + k'u_{n+1}^{\beta}) \circ d_n^{C_{\bullet}} = 0 $. If $Z' \cap \mathrm{Supp}(C_1) \neq  \emptyset $ then the $\HQ^{(1)}$ exactness of $C_{\bullet}$ implies that $h + k'u_{n+1}^{\beta} = h'd_{n+1}^{C_{\bullet}}$  for some morphism $h'$ and finally we obtain $g = g' \circ d_n^{E_{\bullet}}$ where $g' := (h' \enspace k')$. If on the contrary $Z' \cap \mathrm{Supp}(C_1) = \emptyset$ then  the strong $\HQ^{(1)}$ exactness of $C_{\bullet}$ implies that necessarily $h=0$ so we can take $h'=0$ and the claim holds as well in that case.

  Finally, the fact that $\Cbeta$ is strongly $\HQ^{(1)}$ exact follows from the fact that both $C_{\bullet}$ and $D_{\bullet}$ are strongly  $\HQ^{(1)}$ exact and that the components of the differentials of $\Cbeta$ not appearing in those of $C_{\bullet}$ or $D_{\bullet}$ are of the form $\eta_i \otimes \mathrm{id}$ (by construction of $\Cbeta$). 
  \end{proof}

    \section{$q$-characters as Euler characteristics}
    \label{sec : q characters as Euler char}

 For every vertex $x := (i,p) \in \ZQz$, we denote by $Y_{i,p}$ the isomorphism class of the indecomposable object $Y(x)$ in the Grothendieck group (ring) of $\HQ$.  In particular, with the notations of Section~\ref{sec : the subcategory CQone}, we have in $K_0(\HQ^{(1)})$
 $$ [Y(x_i)] =  Y_{i,\xi(i)}  \qquad [Y(\tau x_i)] = Y_{i, \xi(i)-2}   $$
 for each $i \in I$.  We also denote by $f_i := [F(\tau x_i)] \in K_0(\HQ^{(1)})$ for every $i \in I$. 
With these notations, we can state the second main result of this paper. 

 \begin{thm}  \label{thm : Euler characteristics vs q-characters}
     Assume $\mathfrak{g}$ is of type $A_n, n \geq 1$ or $D_n, n \geq 4$ and let $Q$ be an arbitrary orientation of the Dynkin diagram of $\mathfrak{g}$. Then for each $\beta \in \bdelta_+ \sqcup \Pi_-$, we have that 
     $$ \chi \left( C_{\bullet}[\beta]  \right)_{| \forall i, f_i := -1} = \tchi_q \left( L[\beta] \right) .  $$
 \end{thm}

 The proof will consist in showing that the (truncated) $q$-characters of certain simple modules in the HL category $\Cxi^{(1)}$ satisfy certain three-term identities. This will be done by establishing the existence of certain short exact sequences in $\Cxi^{(1)}$.

  We fix a simple Lie algebra $\mathfrak{g}$ of type $A_n , n \geq 1$ or $D_n , n \geq 4$ as well as a height function $\xi$ adapted to an (arbitrary) orientation $Q$ of the Dynkin diagram of $\mathfrak{g}$. We denote by $\theta$ the highest root of $\mathfrak{g}$. 
 Recall the notations $M(\beta)$ and $I(\beta)$ from Section~\ref{sec : R and i(beta)}.
 For each $\beta$ in the positive orthant of the root lattice we denote by $\mathfrak{m}_{\beta} \in \YZ$ the dominant monomial given by 
 $$ \mbeta := \prod_{i \in I} Y_{i, \xi(i)-2}^{c_i} Y_{i, \xi(i)}^{d_i} . $$
 We begin with the following lemma. 

 \begin{lem} \label{lem : monomial in q-character}
  Let $\beta \leq \theta$ and $i,j \in (\Qbeta)_0$ such that $j$ is a sink in $\Qbeta$ and $i \rightsquigarrow j$ in $Q$. Assume also that $a_i < a_j $ and $a_k = a_j$ for all $k$ such that $i \rightsquigarrow k \rightsquigarrow j$. 
  Then the Laurent monomial 
 \begin{equation} \label{eq : monomial in L(mbeta)}
    \mbeta \times  A_i^{-1} \prod_{ \substack{k \neq i  \\ i \rightsquigarrow k \rightsquigarrow j}  }A_k^{- a_j}
 \end{equation}    
 appears in the  (truncated) $q$-character of $L(\mbeta)$. 
 \end{lem} 

  \begin{proof}
   The assumptions imply that $\mbeta$ can be written under the form 
   $$\mbeta =  Y_{\tau x_j}^{a} Y_{x_i}^b \mathfrak{m'} $$
   with $0 \leq b < a = a_j$ and $\mathrm{m}'$ is a dominant monomial that is not divisible by $Y_{\tau x_k}$ nor $Y_{x_k}$ for any $k$ such that $i \rightsquigarrow k \rightsquigarrow j$. The simple module $L(\mbeta)$ is then a Jordan-H\"older component of the tensor product $L(Y_{\tau x_j})^{\otimes a} \otimes L(Y_{x_i})^{\otimes b} \otimes L$ where $L$ is the standard module with highest $l$-weight $\mathfrak{m}'$. It is straightforward to check that there is no dominant monomial $\mathfrak{m}$ satisfying 
   $$   \mbeta \times  A_i^{-1} \prod_{ \substack{k \neq i  \\ i \rightsquigarrow k \rightsquigarrow j}  }A_k^{- a_j}  \enspace \preccurlyeq_N  \enspace  \mathfrak{m}  \enspace \prec_N \enspace \mbeta  $$
    where $\preccurlyeq_N$ denotes the Nakajima ordering. Therefore, ~\eqref{eq : monomial in L(mbeta)} cannot appear in the truncated $q$-character of any simple component of $L(Y_{\tau x_j})^{\otimes a} \otimes L(Y_{x_i})^{\otimes b} \otimes L$ other than $L(\mbeta)$.
  \end{proof}

   We now fix $\beta \in R$ and assume that the desired result has been established for all elements $\gamma \in R$ of height strictly smaller than that of $\beta$.

  \begin{prop} \label{prop : tensor product is not simple}
    Let $i := i(\beta)$. Then the tensor product $L(\mbeta) \otimes L(Y_{x_i})$ is not irreducible.
  \end{prop}

\begin{proof}

 Assume $i$ is a sink in $\Qbeta$ with $a_i= 1$. Then $\mbeta$ is of the form $Y_{\tau x_i} \mathfrak{m}'$ where $\mathfrak{m}'$ involves only $ Y_{x_j} Y_{x_j} , j \neq i$.  If $L(\mbeta) \otimes L(Y_{i, \xi(i)})$ were simple then we would have 
 $$ L(\mbeta) \otimes L(Y_{i, \xi(i)}) \simeq L(Y_{\tau x_i}Y_{x_i} \mathfrak{m}') \simeq L(Y_{\tau x_i}Y_{x_i}) \otimes L(\mathfrak{m}') .   $$
 In particular this implies 
 $$ \tchi_q(L(\mbeta)) = Y_{\tau x_i} \tchi_q(L(\mathfrak{m}')) .  $$
  Because of the form of $\mathfrak{m}'$, we see that the monomial $ \mbeta A_{i}^{-1} = Y_{\tau x_i}\mathfrak{m}'A_i^{-1}$ cannot appear on the right hand side; however, it has to appear on the left hand side by Lemma~\ref{lem : monomial in q-character}, hence a contradiction. 

   Assume now that $i = i(\beta)$ is such that there is an oriented path 
   $$  k \rightarrow \cdots \rightarrow i \rightarrow i' \rightarrow \cdots \rightarrow j $$
   in $\Qbeta$ with $a_k = \cdots = a_i=1$,  $a_{i'} = \cdots = a_j = 2$ and $j$ is  a sink in $\Qbeta$ so that $\mbeta$ is of the form $Y_{\tau x_j}^2Y_{x_i}\mathfrak{m}'$. It is straightforward to check that 
   $$ \mbeta Y_{x_i} = \mathfrak{m}_{\gamma} \mathfrak{m}'' \qquad \gamma := \beta - \sum_{j \rightsquigarrow i} \alpha_j  $$
   where $\mathfrak{m}''$ is a dominant monomial dividing $\mbeta$ and containing only variables of the form $Y_{x_p}$ or $Y_{\tau x_q}$ with $q$ such that there is no oriented path between $q$ and $i$. In particular $\gamma$ has height strictly less than that of $\beta$ and thus by the induction assumption we have that $L(\mathfrak{m}_{\gamma}) \simeq L[\gamma]$. As $i \notin \mathrm{Supp}(\gamma)$ (because $a_i = 1$), this implies in particular that the  monomial 
   \begin{equation} \label{eq : monomial}
        (A_j^{-1} \cdots A_{i'}^{-1})^2 A_i^{-1}
   \end{equation}
   does not appear in the  renormalized (truncated) $q$-character of $L(\mathfrak{m}_{\gamma})$. On the other hand, if $L$ is any fundamental representation whose highest $l$-weight divides  $\mathfrak{m}''$, then the renormalized truncated $q$-character of $L$ is either $1$ (if $L \simeq L(Y_{x_p})$ for some $p$) or contains only monomials divisible by $A_q^{-1}$ if $L \simeq L(Y_{\tau x_q})$.  But as mentioned above, there is no oriented path between $i$ and $q$ and hence the  renormalized $q$-character of $L$ cannot contain~\eqref{eq : monomial} nor any monomial dividing it. Consequently, we obtain that the monomial~\eqref{eq : monomial} cannot appear in the renormalized $q$-character of $L(\mathfrak{m}_{\gamma} \mathfrak{m}'')$. So, if $L(\mbeta) \otimes L(Y_{x_i})$ was simple, we would have
   $$ L(\mbeta) \otimes L(Y_{x_i}) \simeq L(\mbeta Y_{x_i}) \simeq L(\mathfrak{m}_{\gamma} \mathfrak{m}'')$$
   so that the monomial~\eqref{eq : monomial} would not appear in the renormalized $q$-character of $L(\mbeta)$. This is a contradiction by Lemma~\ref{lem : monomial in q-character}.
   It remains to deal with the case where $M(\beta) = \emptyset$, i.e. $a_k=2$ for all $k \in \mathrm{Supp}(\beta)$. In that case we obviously have that $\beta = 2 \gamma$ where $\gamma$ is a multiplicity-free positive root and we get that $\mbeta = \mathfrak{m}_{\gamma}^2$. Moreover we obviously have $Q_{\gamma} = \Qbeta$ and by definition of $i(\beta)$ we have that  $i$ is a sink in $Q_{\gamma}$. We then have 
   \begin{align*}
       \bdelta(L(Y_{x_i}) , L(\mbeta)) &= \bdelta(L(Y_{x_i}) , L(\mathfrak{m}_{\gamma}^2)) \\
       &= \bdelta(L(Y_{x_i}) , L(\mathfrak{m}_{\gamma})^{\otimes 2})  \quad \text{as all simple modules in $\Cxi^{(1)}$ are real} \\
       &= 2 \bdelta(L(Y_{x_i}) , L(\mathfrak{m}_{\gamma}))  \quad  \text{by Theorem~\ref{thm : KKOP}(iv)} \\
       & \neq 0
   \end{align*}
   as $L(Y_{x_i})$ and $L(\mathfrak{m}_{\gamma})$ do not commute by what has been done in the first case. This concludes the proof of the proposition. 
\end{proof}

   \begin{prop} \label{prop : length two}
     Let $\beta \leq \theta$ and $i  := i(\beta)$. Then the tensor product $L(\mbeta) \otimes L(Y_{x_i})$ is of length two in the HL category $\Cxi^{(1)}$.   
   \end{prop}

\begin{proof}
Let us first assume that $a_i=1$. In that case, we are going to show that $\bdelta(L(Y_{x_i}) , L(\mbeta)) = 1$, which will imply the desired statement by Theorem~\ref{thm : KKOP}(ii). As Proposition~\ref{prop : tensor product is not simple} implies that $\bdelta(L(Y_{x_i}) , L(\mbeta)) \neq  0$, it actually suffices to prove that $\bdelta(L(Y_{x_i}) , L(\mbeta)) \leq  1$. 
If $i$ is a sink in $\Qbeta$ then $\mbeta = Y_{\tau x_i} \mathfrak{m}'$ where $\mathfrak{m}'$ is a dominant monomial divisible only by $Y_{x_p}$ or $Y_{\tau x_q}$ such that there is no oriented path from $i$ to $q$. Consequently, we can choose an appropriate PBW order on the collection of fundamental representations in $\Cxi^{(1)}$ so that 
$$ L(\mbeta) \simeq \mathrm{hd}(L(Y_{\tau x_i}) \otimes L(\mathfrak{m}')) .   $$
Then, Theorem~\ref{thm : Fujita} implies that $L(Y_{x_i})$ strongly commutes with  any fundamental representation whose highest $l$-weight divides $\mathfrak{m}'$ and hence $L(Y_{x_i})$ also strongly commutes with $L(\mathfrak{m}')$. In particular we get $\bdelta(L(Y_{x_i}) , L(\mathfrak{m}'))= 0$ and hence  we obtain
$$  \bdelta(L(Y_{x_i}) , L(\mbeta)) =  \bdelta(L(Y_{x_i}) ,  \mathrm{hd}( L(Y_{\tau x_i}) \otimes  L(\mathfrak{m}')) \leq \bdelta(L(Y_{x_i}) , L(Y_{\tau x_i})) = 1 . $$
Thus the claim is proved in that case. If there is no sink in $M(\beta)$, then as explained in the proof of Proposition~\ref{prop : tensor product is not simple} we have $\mbeta = Y_{\tau x_j}^2Y_{x_i} \mathfrak{m}'$. For the same reasons as in the previous case we have that $L(\mathfrak{m}')$ strongly commutes with $L(Y_{x_i})$ and thus using Theorem~\ref{thm : KKOP} (v) we obtain 
$$ \bdelta(L(\mbeta), L(Y_{x_i})) \leq \bdelta(L(Y_{\tau x_j}^2Y_{x_i}) , L(Y_{x_i})) = \max \left( 0,  2 \bdelta(L(Y_{\tau x_j}),L(Y_{x_i})) - 1 \right)= 1  $$
where the last equality follows from Theorem~\ref{thm : Fujita} as there is an oriented path from $i$ to $j$ in $Q$. Consequently the desired statement is proved in the case where $a_{i(\beta)} = 1$. 

 Assume now that $a_{i(\beta)} = 2$ which means that $a_k=2$ for all $k \in \mathrm{Supp}(\beta)$. Then as mentioned above we have $\beta = 2 \gamma$ and the arguments from the first case above show that  $\bdelta(L(\mathfrak{m}_{\gamma},L(Y_{x_i})) = 1$, hence by Theorem~\ref{thm : KKOP}(ii) a short exact sequence in $\Cxi^{(1)}$
 $$ 0 \rightarrow  A \rightarrow  L(\mathfrak{m}_{\gamma}) \otimes L(Y_{x_i}) \rightarrow B \rightarrow 0  $$
 with $A$ and $B$ (real) simple modules both strongly commuting with $L(\mathfrak{m}_{\gamma})$ (as well as $L(Y_{x_i})$). In particular, the tensor product $L(\mathfrak{m}_{\gamma})^{\otimes 2} \otimes L(Y_{x_i})$ has exactly two simple components, respectively isomorphic to $L(\mathfrak{m}_{\gamma}) \otimes A$ and $L(\mathfrak{m}_{\gamma}) \otimes B$. In conclusion, we have that $L(\mbeta) \otimes L(Y_{x_i})$ is of length two in this case as well, which concludes the proof of the proposition. 
   \end{proof}

\begin{prop} \label{prop : identity of q-characters}
Denote by $X_i := Y_{\tau x_i}Y_{x_i}$ for each $i \in I$. Let $\beta \leq \theta$ and $i \in I(\beta)$. Then there is an identity of truncated $q$-characters 
$$ \tchi_q(L(\mbeta)) \cdot Y_{x_i} = X_i  \cdot \tchi_q(L(\mathfrak{m}_{\beta - \boldsymbol{\dim}I_i^{\beta}}) + Y(i, \beta) X(i, \beta)  \tchi_q(L(\mathfrak{m}_{\beta - \boldsymbol{\dim}P_i^{\beta}})) $$
where $Y(i, \beta) := [H(i, \beta)]$ and $X(i, \beta) := [K(i, \beta)]$ with the notations of Proposition~\ref{prop : Serre tilting  of leading object}. 
\end{prop}

\begin{proof}
    Assume that $i$ is a sink in $\Qbeta$ and $a_i=1$. By Proposition~\ref{prop : length two}, we have an identity in $K_0(\Cxi^{(1)})$:
    $$ [L(\mbeta)] [L(Y_{x_i})] = [L(\mbeta Y_{x_i})] + [M] $$
    where $M$ is a  simple module in $\Cxi^{(1)}$ whose highest $l$-weight needs to be determined. Before doing so, we can already note that, using the same notations as above, we have 
    $$ \mbeta Y_{x_i} = X_i \mathfrak{m}_{\beta - \sum_{j \rightsquigarrow i} \alpha_j} .  $$
    By Lemma~\ref{lem : monomial in q-character}, the monomial $\mbeta A_i^{-1}$ must appear in the (truncated) $q$-character of $L(\mbeta)$. However, as $a_i = 1$ we have that $i \notin \mathrm{Suppp}(\beta - \sum_{j \rightsquigarrow i} \alpha_j)$. Hence by the induction assumption, $ \mathfrak{m}_{\beta - \sum_{j \rightsquigarrow i} \alpha_j} \cdot A_i^{-1} $ cannot appear in the truncated $q$-character of $L(\mathfrak{m}_{\beta - \sum_{j \rightsquigarrow i} \alpha_j}) $ and thus has to appear in that of $M$. Now, we note that
    $$ \mbeta Y_{x_i} A_i^{-1} = \prod_{i \rightarrow j} Y_{x_j} \cdot \prod_{i \leftarrow j}Y_{\tau x_j}  \mathfrak{m}'  $$
    which is obviously dominant. Therefore it has to be the highest $l$-weight of $M$. We are now going to relate it to $\mathfrak{m}_{\beta-\alpha_i}$. By definition, we have that
    $$ \mathfrak{m}_{\beta - \alpha_i} = \prod_{k \in I} Y_{\tau x_k}^{c'_k} Y_{x_k}^{d'_k}  $$
    with $c'_i = c_i -1 = 0 = d'_i$ and $c'_j-d'_j = c_j-d_j + 1$ if $i \leftarrow j$, $c'_j=c_j, d'_j=d_j$ otherwise. In other words, the only difference between $\mathfrak{m}'$ and $\mathfrak{m}_{\beta - \alpha_i}$ is at the vertices $j$ such that $i \leftarrow j$, where $c'_j = c_j +1, d'_j=d_j=0$ if $c_j \geq 0$, while $c'_j=c_j=0, d'_j = d_j-1$ if $d_j>0$. Thus we see that 
    $$ \mathfrak{m}' \cdot \prod_{ \substack{i \leftarrow j \\ c_j-d_j \geq 0 }} Y_{\tau x_j} = \mathfrak{m}_{\beta - \alpha_i}  \cdot   \prod_{ \substack{i \leftarrow j \\  c_j-d_j<0}} Y_{x_j} . $$
    This yields 
    $$ \mbeta Y_{x_i} A_i^{-1} = \mathfrak{m}_{\beta-\alpha_i} \cdot  \prod_{i \rightarrow j} Y_{x_j} \cdot  \prod_{ \substack{i \leftarrow j \\  c_j-d_j<0}} X_j . $$
    Finally, we have that for each $j$ with $i \rightarrow j$, the fundamental representation $L(Y_{x_j})$ strongly commutes with $L(\mathfrak{m}_{\beta-\alpha_i})$ because any $l$ such that $j \rightsquigarrow l$ necessarily lies outside the support of $\beta$.  Hence we obtain 
    $$ M \simeq L(\mbeta Y_{x_i}A_i^{-1}) \simeq \bigotimes_{i \rightarrow j}L(Y_{x_j}) \otimes \bigotimes_{\substack{i \leftarrow j \\ a_j < \sum_{j \rightarrow k} a_k}} L(Y_{\tau x_j}Y_{x_j}) \otimes L(\mathfrak{m}_{\beta - \alpha_i})  $$
    and the proposition is proved in this case. The proof in the other cases is done in a similar way, combining with the same calculations as those performed throughout the proof of Proposition~\ref{prop : Serre tilting  of leading object}.  
\end{proof}

 \begin{proof}[\textbf{Proof of Theorem~\ref{thm : Euler characteristics vs q-characters}}]
   It follows from the constriction of $\Cbeta$ (cf. Section~\ref{sec : construction of Cbeta} and more specifically Proposition~\ref{prop : morphism of chain complexes}) that there is a distinguished triangle in $\KQ^{(1)}$ 
   $$ D_{\bullet} \rightarrow \Cbeta \otimes \mathbf{H}_i \rightarrow C_{\bullet} \rightarrow D_{\bullet}[1] $$
   where $C_{\bullet}$ and $D_{\bullet}$ are respectively given by  
   $$C_{\bullet} :=  \mathbf{K}_i^{\otimes \epsilon_i} \otimes C_{\bullet}[\beta - \boldsymbol{\dim}I_i^{\beta}] \qquad  D_{\bullet} :=  \mathbf{F}(i, \beta) \otimes  \mathbf{H}(i, \beta) \otimes  \mathbf{K}(i, \beta) \otimes  C_{\bullet}[\beta - \boldsymbol{\dim}P_i^{\beta}] .   $$
   This implies that the Euler characteristics of the chain complexes $\Cbeta$ (after specializing the classes of $F_i , i \in I$ to $-1$) satisfy the same identities as those appearing in Proposition~\ref{prop : identity of q-characters}. It remains to check that $\chi(\Cbeta)_{\mid [F_i] = -1}$ agree with $\tchi_q(L(\mbeta))$ when $\beta$ is a simple root, which is straightforward. As a consequence of this, we can conclude that $L(\mbeta) \simeq L[\beta]$ in the Hernandez-Leclerc category $\Cxi^{(1)}$. In other words, the simple module of dominant monomial $\mbeta$ categorifies the cluster monomial of $d$-vector $\beta$. Indeed, the identity $\chi(\Cbeta)_{\mid [F_i] = -1} = \tchi_q(L(\mbeta))$ implies that the lowest $l$-weight monomial $\mathfrak{m}'_{\beta}$ of $\tchi_q(L(\mbeta))$ is related to $\mbeta$ by 
   $$  \mathfrak{m}'_{\beta} = \mbeta \times \prod_{i \in I}A_i^{-a_i} . $$
   As it follows from the results of Kashiwara-Kim-Oh-Park that all simple modules in $\Cxi^{(1)}$ are cluster monomials,  the discussion above implies in particular that if $\beta$ is a positive root, then the class of $L(\mbeta)$ in $K_0(\Cxi^{(1)})$ has to be a cluster variable, namely that with $d$-vector $\beta$. With the previous notations this means that $L(\mbeta)$ and $L[\beta]$ have same isomorphism class in $K_0(\Cxi^{(1)})$ and hence are isomorphic given that they are simple modules.  This finishes the proof of Theorem~\ref{thm : Euler characteristics vs q-characters}. 
   In conclusion, in addition to Theorem~\ref{thm : Euler characteristics vs q-characters}, we proved the following, agreeing with results by Brito-Chari \cite{BC} in type $A_n$. 

    \begin{thm}
        Assume $\mathfrak{g}$ is of type $A_n$ with $n \geq 1$ or $D_n$ with $n \geq 4$, and let $Q$ be an arbitrary orientation  of the Dynkin diagram of $\mathfrak{g}$. Choose a height function $\xi :  Q_0 \rightarrow \mathbb{Z}$ adapted to $Q$. Given a positive root $\beta = \sum_{i \in I}a_i \alpha_i$, define $b_i := a_i - \sum_{i \rightarrow j}a_j, c_i := (b_i)_+, d_i := (b_i)_-$ and consider the dominant monomial $\mbeta := \prod_{i \in I} Y_{i, \xi(i)-2}^{c_i}Y_{i, \xi(i)}^{d_i}$. Then  the class of $L(\mbeta)$ in $K_0(\Cxi^{(1)})$ is the cluster variable $x[\beta]$ of $\AQ$. 
    \end{thm}
 \end{proof}

\end{document}